\PassOptionsToPackage{dvipsnames,table}{xcolor}
\documentclass[final,onefignum,onetabnum,pagebackref]{siamart251216}
\usepackage{graphicx} 
\usepackage{algorithm, algpseudocode, algorithmicx}
\usepackage{amssymb}
\usepackage{booktabs}
\usepackage{enumitem}
\usepackage{float}
\usepackage{mathtools}
\usepackage{placeins}
\usepackage{nicefrac}       
\usepackage{xspace}
\usepackage{url}
\usepackage{stmaryrd}
\usepackage{cite}
\usepackage{subfig}

\usepackage[most]{tcolorbox}
\newcommand{\actionbox}[1]{\begin{tcolorbox}[colback=white,colframe=black,width=\columnwidth,boxsep=5pt,arc=4pt]
    \centering \emph{#1}
\end{tcolorbox}}

\let\norm\relax
\DeclarePairedDelimiter\norm{\lVert}{\rVert}
\DeclareMathOperator{\err}{err}

\newcommand{\f}[2]{\frac{#1}{#2}}
\newcommand{\fl}[1]{\mathrm{fl}#1}
\newcommand{\m}[1]{\mathbb{#1}}
\renewcommand{\b}[1]{\boldsymbol{#1}}
\newcommand{\wb}[1]{\widehat{\boldsymbol{#1}}}
\newcommand{\tb}[1]{\tilde{\boldsymbol{#1}}}
\newcommand\numberthis{\addtocounter{equation}{1}\tag{\theequation}}
\newcommand{\warn}[1]{\emph{#1}}
\newcommand{\order}{\mathcal{O}}
\newcommand{\mat}[1]{\b{#1}}
\renewcommand{\vec}[1]{\b{#1}}
\renewcommand{\top}{\intercal}

\renewcommand{\eqref}{\cref}
\renewcommand{\tilde}{\widetilde}
\renewcommand{\hat}{\widehat}

\newsiamthm{fact}{Fact}
\newsiamthm{conjecture}{Conjecture}
\newsiamthm{remark}{Remark}
\newsiamthm{assumption}{Assumption}
\newsiamthm{informaltheorem}{Informal Theorem}
\crefname{section}{Section}{Sections}
\crefname{subsection}{Section}{Sections}
\crefname{subsubsection}{Section}{Sections}

\renewcommand*{\backref}[1]{}
\renewcommand*{\backrefalt}[4]{%
  \ifcase #1 %
  (No citations.)
  \or
  (Cited p.~#2.)
  \else
  (Cited pp.~#2.)
  \fi
}

\title{Numerical instabilities in the Kaczmarz method \\ and stabilization by iterative refinement\thanks{Date: \today.
\funding{ENE acknowledges support from the Miller Institute for Basic Research in Science, University of California Berkeley. MD was supported in part by NSF CAREER Grant CCF-233865 and a Google ML and Systems Junior Faculty Award. DN and AX were partially supported by NSF DMS 2408912. This work was done in part while MD and ENE were visiting the Simons Institute for the Theory of Computing.}}}

\author{Micha\L\ Derezi\'nski\thanks{Division of Computer Science and Engineering, University of Michigan  (\email{derezin@umich.edu}).}
\and Ethan N.\ Epperly\thanks{Department of Mathematics, University of California, Berkeley (\email{eepperly@berkeley.edu}).}
\and Deanna Needell\thanks{Department of Mathematics, University of California, Los Angeles (\email{deanna@math.ucla.edu, alexxue@g.ucla.edu}).}
\and Alexander Xue\footnotemark[4]
}

\begin{document} 
\maketitle

\begin{abstract}
The randomized Kaczmarz method and its accelerated variants are a powerful class of iterative methods for solving large-scale linear systems, offering guaranteed convergence with low per-iteration cost.
However, their numerical stability remains poorly understood.
In this work, we investigate the stability properties of both classical and accelerated randomized Kaczmarz methods, with an emphasis on how error propagates across iterations and interacts with acceleration.
We show that both classical and accelerated randomized Kaczmarz fail to be forward stable.
To address this issue, we propose the integration of iterative refinement into randomized Kaczmarz frameworks.
We demonstrate that refinement can effectively control error accumulation and recover high-accuracy solutions, even when the system is ill-conditioned. 
Numerical experiments corroborate our theoretical findings and illustrate the practical benefits of combining refinement with both classical and accelerated Kaczmarz methods.
\end{abstract}



\section{Introduction}

The randomized Kaczmarz method \cite{strohmer2007randomizedkaczmarzalgorithmexponential} is among the most widely studied numerical algorithms in the twenty-first century.
The method iteratively solves a system of linear equations $\mat{A}\vec{x} = \vec{b}$ by randomly selecting a single equation and projecting the current iterate onto them:
\begin{equation} \label{eq:kaczmarz}
    \vec{x}_{k+1} = \vec{x}_k + \frac{b_{r(k)} - \langle \vec{a}_{r(k)} , \vec{x}_k \rangle}{\langle \vec{a}_{r(k)} , \vec{a}_{r(k)} \rangle} \vec{a}_{r(k)} \quad \text{for } k = 0,1,2,\ldots.
\end{equation}
Here, $b_i$ and $\vec{a}_i^\top$ denote the $i$th entry or row of $\vec{b}$ or $\mat{A}$, and 
\begin{equation} \label{eq:random-row}
    r(k) = i \quad \text{with probability } \norm{\vec{a}_i}^2 / \norm{\mat{A}}_{\rm F}^2
\end{equation}
is a randomly selected index.
Throughout this article, $\norm{\cdot}$ denotes the vector $\ell_2$ norm and its induced operator norm, and $\norm{\cdot}_{\rm F}$ denotes the Frobenius norm.
The Kaczmarz iteration \cref{eq:kaczmarz} dates back to the work of Stefan Kaczmarz in 1937 \cite{Kac37} and received significant attention in image reconstruction in the 1970s \cite{GBH70}.
Randomized Kaczmarz can be interpreted as a weighted version of stochastic gradient descent with a specific choice of step size \cite{NSW16}.
As such, the solution of linear systems by randomized Kaczmarz has received attention not only for its efficiency in solving systems, but also as a model problem for studying phenomena associated with stochastic gradient methods such as implicit regularization \cite{MNR15}, the effect of the minibatch size \cite{BCW25}, the interpolation regime \cite{derezinski2026last}, and catastrophic forgetting \cite{EMW+22}.

\subsection{Contributions}

Despite the significant interest in (randomized) Kaczmarz methods and the even wider interest in stochastic gradient methods, the behavior of randomized Kaczmarz in finite-precision arithmetic appears to have avoided careful study.
We ask:
\actionbox{Is randomized Kaczmarz numerically stable? That is, if run to convergence, does it produce results of comparable accuracy to standard direct methods?}

\subsubsection{First result: Instability of randomized Kaczmarz}
Our first contribution is to show that, perhaps surprisingly, the answer to this question is no: the randomized Kaczmarz algorithm is not numerically stable in general.
Even when applied to a consistent system (i.e., one with $\vec{b} \in \operatorname{range}(\mat{A})$), we demonstrate that the method can stagnate far above the level of accuracy achieved by a direct solve when the problem is ill-conditioned.
This is demonstrated in \cref{fig:exp_1e4}, which shows an example in single precision (unit roundoff $u\approx 10^{-8}$) where standard randomized Kaczmarz (blue curve) stagnates at a relative forward error $\norm{\vec{x} - \hat{\vec{x}}} / \norm{\vec{x}}$ that is orders of magnitude worse than MATLAB's ``\texttt{A\textbackslash{}b}'' direct solver (the horizontal line).
The model problem is randomly generated and has a Demmel condition number
\begin{equation} \label{eq:demmel}
    \tilde{\kappa}(\mat{A}) = \norm{\mat{A}}_{\rm F} \norm{\mat{A}^{-1}}
\end{equation}
of $\tilde \kappa (\mat A) = 10^4$.
Since the forward error of randomized Kaczmarz is much larger than the forward error of a direct method, we conclude that randomized Kaczmarz is not forward stable.
\Cref{sec:notation} describes the notion of forward stability more precisely, and \cref{sec:experiment} provides details on the setup and other experiments.

\begin{figure}
    \centering
    \subfloat[$\tilde \kappa = 10^4$]
    {\includegraphics[width=0.48\linewidth]{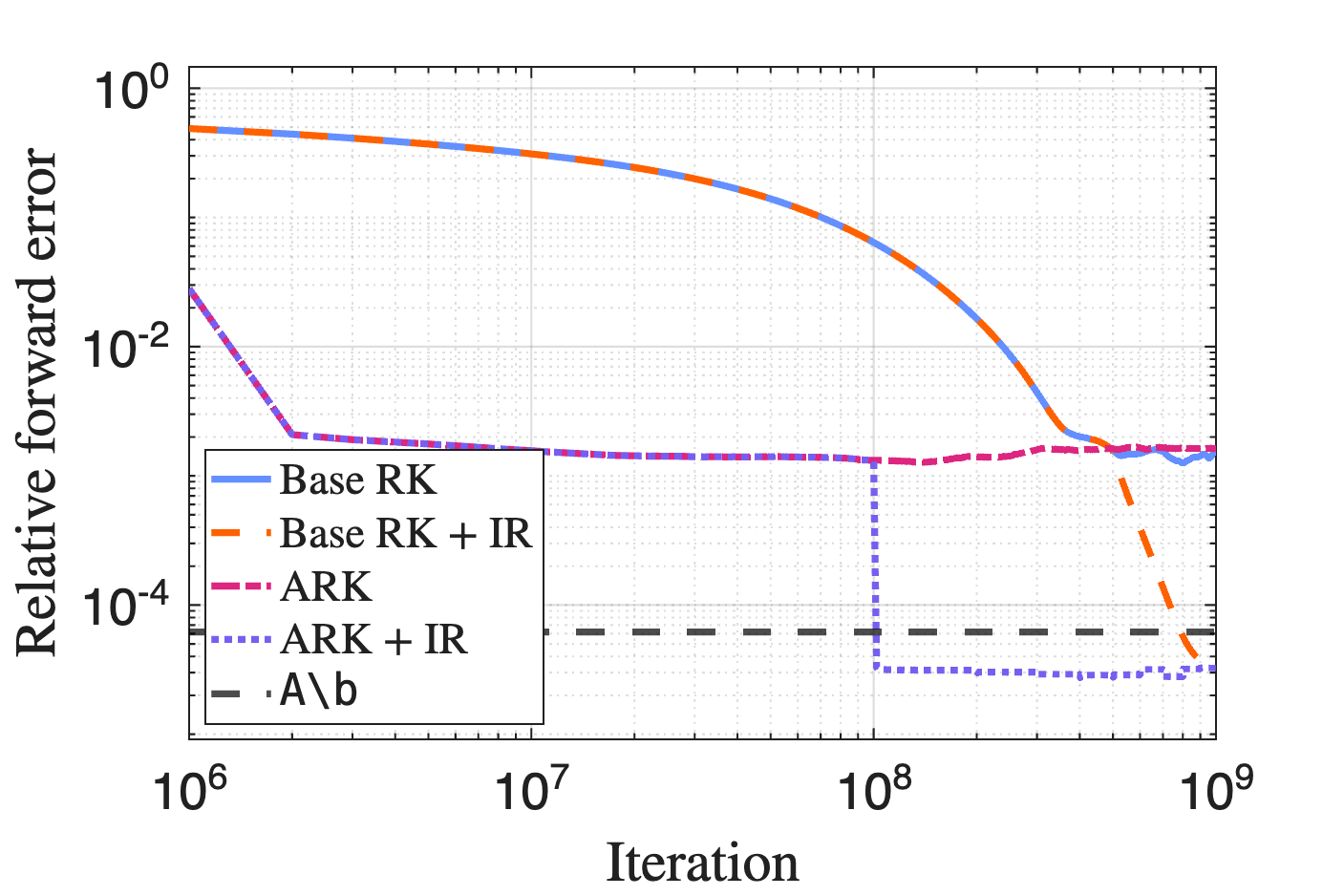}
    \label{fig:exp_1e4}}
    \subfloat[$\tilde \kappa = 10^5$]
    {\includegraphics[width=0.48\linewidth]{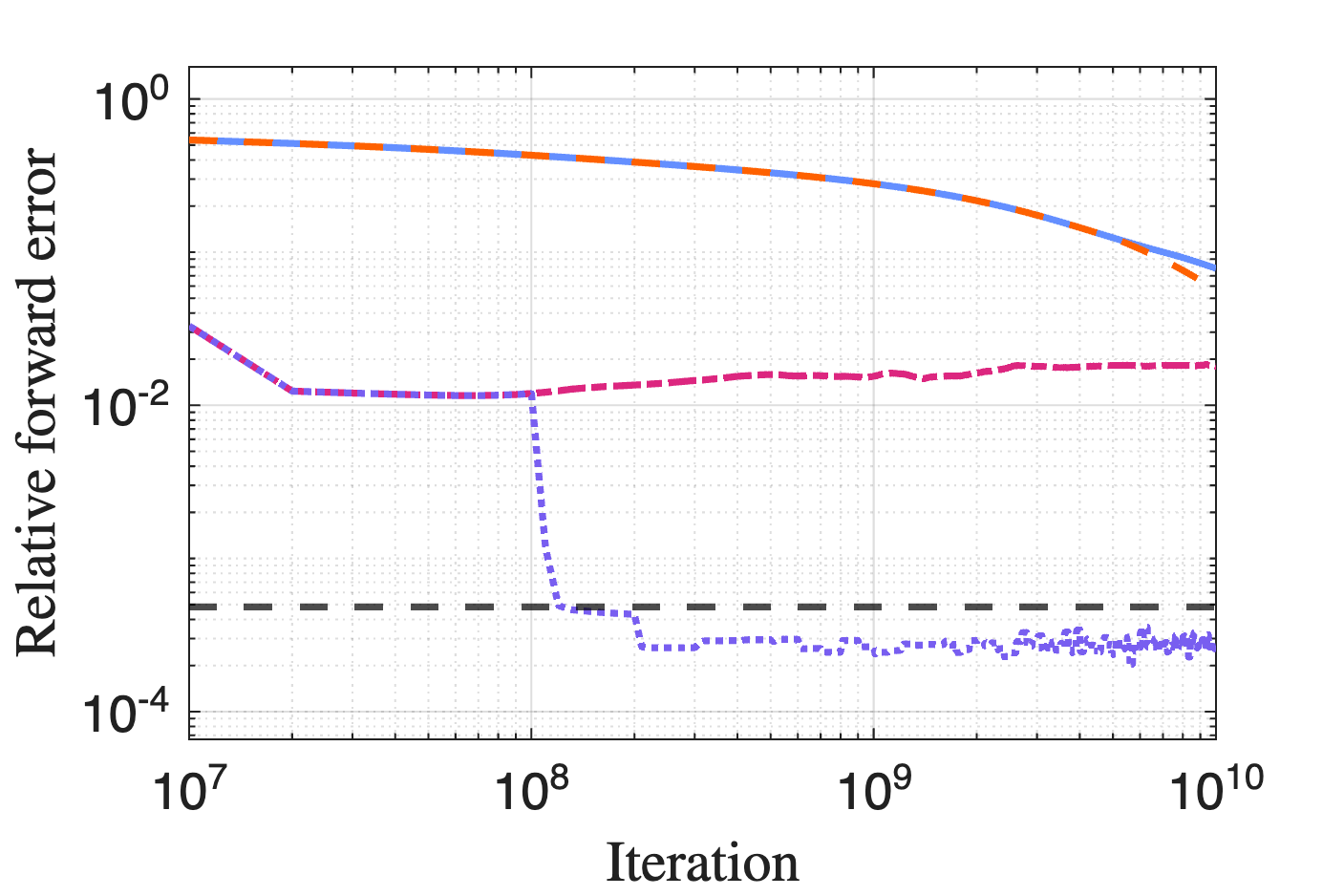}
    \label{fig:exp_1e5}}
    \caption{Error versus iteration for randomized Kaczmarz (blue), randomized Kaczmarz with iterative refinement (orange), accelerated randomized Kaczmarz (ARK, pink), and ARK with iterative refinement (purple) on a randomly generated consistent system with exponentially distributed singular values of size $n=500$.
    The Demmel condition number is $\tilde \kappa(\mat{A}) = 10^4$ in (a) and $\tilde \kappa(\mat{A}) = 10^5$ in (b). The horizontal dashed line is the error of MATLAB's ``\texttt{A\textbackslash{}b}'' direct solver.}
    \label{fig:exp}
\end{figure}

If it is not always forward stable, how accurate is the randomized Kaczmarz method?
We answer this question by proving the following theorem: 

\begin{theorem} [Suboptimal stability of randomized Kaczmarz] \label{thm:base_kz}
Let $\mat A \in \m R^{n \times n}$ be an invertible matrix, let $\vec{b} \in \m R^{n}$ be a vector, and let $\vec{x} \coloneqq \mat{A}^{-1}\vec{b}$.
Let $\widehat{\b x}_{k}$ be the output of $k$ iterations of randomized Kaczmarz on initial input $\widehat{\b x}_{0}$ executed in floating-point arithmetic with precision $u$, and assume $ \norm{\wb x_0} \le \norm {\b x}$.
Finally, assume that $\tilde\kappa(\mat{A})^2 u \ll 1$. 
Then it holds that
$$\m E \| \widehat{\b x}_k - \b x\| \le  \left(1 - \f {1} {2\tilde{\kappa}(\mat{A})^{2}}\right) ^k \norm{\wb x_0 - \b x} + C \| \b x\| \cdot \tilde{\kappa}(\mat{A})^2 u$$
for some polynomial $C = C(n)$.
In particular, if we run for $k = \order(\tilde{\kappa}(\mat{A})^2 \log(1/u))$ total iterations, then the output satisfies
\begin{equation*}
    \m E \| \widehat{\b x}_k - \b x\| \lesssim \norm{\vec{x}} \cdot  \tilde{\kappa}(\mat{A})^2u.
\end{equation*}
\end{theorem}

The symbol $\ll$ has a precise meaning in this context, as detailed in \cref{sec:notation}.
We remind the reader that the unit roundoff $u$ is a measure of the size of rounding errors.
It is a property of a system of floating point numbers.
In double precision, we have $u\approx 10^{-16}$, and in single precision, we have $u\approx 10^{-8}$.

\Cref{thm:base_kz} suggests that the error achieved by randomized Kaczmarz is roughly $\tilde{\kappa}(\mat{A})^2 u$.
In contrast, standard direct methods achieve an error of roughly $\tilde{\kappa}(\mat{A})u$.
The larger $\tilde{\kappa}(\mat{A})^2 u$ error of randomized Kaczmarz has a simple explanation:
Each step of the Kaczmarz algorithm incurs an error of size $\order(u)$, and $k$ steps of randomized Kaczmarz damps errors by a factor of roughly $(1 - 2/\tilde{\kappa}(\mat{A})^2)^k$.
Therefore, the total error is roughly
\begin{equation*}
    \sum_{t=0}^\infty \order(u) \cdot (1 - 2/\tilde{\kappa}(\mat{A})^2)^k = \order(\tilde{\kappa}(\mat{A})^2u).
\end{equation*}
This informal argument mirrors our formal proof of \cref{thm:base_kz} in \cref{sec:kz}.

\subsubsection{Second result: Stabilization by iterative refinement}

When the matrix $\mat{A}$ is not too ill-conditioned, randomized Kaczmarz can be stabilized in a straightforward way by combining the method with \emph{iterative refinement}.
The strategy is as follows.
First, run the randomized Kaczmarz method on $\mat{A}\vec{x} = \vec{b}$ until it reaches its maximum achievable accuracy, resulting in an approximate solution $\wb{x}$.
Then, apply randomized Kaczmarz to the residual equation $\mat{A}\vec{e} = \vec{r}$ with $\vec{r} = \vec{b} - \mat{A}\wb{x}$, obtaining an approximate solution $\wb{e}$.
Finally, output $\wb{x} + \wb{e}$ as a numerical approximation to $\vec{x}$.
We emphasize that this entire procedure is executed in a uniform numerical precision (i.e., we do not use higher-precision arithmetic to evaluate the residual $\vec{r}$).
For difficult problems, more than one step of this process may be necessary.

Our second main result is that randomized Kaczmarz with iterative refinement is forward stable:
\begin{theorem}[Stability of randomized Kaczmarz with iterative refinement] \label{thm:base_kz_ir}
Let $\mat A \in \m R^{n \times n}$ be an invertible matrix, let $\b b \in \m R^{n}$ be a vector, and let $\b x = \mat{A}^{-1} \b b.$
Suppose that we perform iterative refinement $t=\log_2(\f 1 {\tilde{\kappa}(\mat{A}) u})$ times, where each refinement executes $\order(\tilde\kappa(\mat{A})^2)$ iterations of randomized Kaczmarz. If $\tilde \kappa(\mat{A})^2 u \ll 1$, then this scheme attains forward stability:
    $$\m E \norm{\wb x_t - \b x } \lesssim \norm{\b x} \cdot \tilde\kappa(\mat{A}) u.$$
   
\end{theorem}

Our analysis is generic and guarantees convergence with iterative refinement for a general solver subroutine that reduces the error by a constant factor, in expectation.
This general result appears in \cref{sec:ir}.
Note that this stability result is conditional on the assumption $\tilde \kappa(\mat{A})^2 u \ll 1$.
This requirement is sensible as, according to \cref{thm:base_kz}, if $\tilde \kappa(\mat{A})^2 u$ is not small, then we cannot guarantee that the initial execution of randomized Kaczmarz reduces the error even by a constant factor.

The fact that iterative refinement stabilizes randomized Kaczmarz is perhaps no surprise---iterative refinement is a classical strategy in numerical analysis that is known for upgrading weakly stable algorithms into more strongly stable ones.
What is striking, however, is just how little it takes here. 
In exact arithmetic, running one step of iterative refinement where the first solve is $k_1$ steps of randomized Kaczmarz and the refinement step is $k_2$ of randomized Kaczmarz is equivalent to simply running $k_1+k_2$ steps of randomized Kaczmarz.
This observation is proven in \cref{lem:ir_kz_exact}.
If the methods are the same in exact arithmetic, how does iterative refinement help?
The difference is bookkeeping: iterative refinement computes the residual $\vec{r} = \vec{b} - \mat{A}\wb{x}$ and tracks updates in a separate vector ``$\vec{e}$" rather than updating ``$\vec{x}$" directly.
Computing the residual $\vec{b} - \mat{A}\wb{x}$ and representing the solution as sum $\vec{x} + \vec{e}$ of two vectors is enough to transform a not-fully stable method into a forward stable one. 


\subsubsection{Third result: Acceleration}

The analysis of Strohmer and Vershynin \cite{strohmer2007randomizedkaczmarzalgorithmexponential} shows that the basic randomized Kaczmarz iteration converges slowly when the matrix is ill-conditioned.
Indeed, in \cref{fig:exp_1e4}, it takes roughly one billion iterations for Kaczmarz to converge to the maximum attainable accuracy, and standard Kaczmarz fails to converge to the maximum attainable accuracy in \cref{fig:exp_1e5} even after ten billion iterations.
To speed up convergence, we can turn to accelerated versions of the Kaczmarz procedure.
Within the context of this paper, it is natural to ask:
\actionbox{Does acceleration help or hurt the stability of randomized Kaczmarz?}

To answer this question, we must specify a particular acceleration scheme, as there are many in the literature \cite{gower2018accelerated, derezinski2025fine, LW13}.
In this work, we focus on the accelerated randomized Kaczmarz (ARK) method from \cite{derezinski2025randomized}, which accelerates convergence by means of a momentum term.
To focus on the underlying numerical issues, we analyze the simplest version of ARK where we use optimal acceleration parameters and, similar to base Kaczmarz, perform single-row updates. 
Details of the ARK method are given in \cref{sec:rak}.


As \cref{fig:exp} demonstrates, ARK converges much faster than randomized Kaczmarz, but it is also not forward stable.
Just as with base Kaczmarz, we prove a suboptimal stability result for ARK showing that the error stagnates at a level not exceeding $\tilde{\kappa}(\mat{A})^2 u$.

\begin{informaltheorem} [Suboptimal stability of ARK] \label{infthm:accelkz}
    Assume that the invertible $\mat A$ satisfies $\tilde \kappa(\mat A)^2 u \ll 1.$ Then, with particular choices of algorithm parameters, the floating-point output $\wb x_k$ of ARK after $k$ iterations satisfies
     $$ \m E \|\widehat{\b x}_k - \b x\| \lesssim \left(1 - \f 1 {4 \tilde \kappa(\mat A) \sqrt{n}} \right)^{k/4}\| \b x \| +  \| \b x \| \cdot \tilde \kappa(\mat A)^2  u. $$
      In particular, if we run for $k = \order(\tilde{\kappa}(\mat{A}) \sqrt n \log(1/u))$ iterations, then $\wb x_k$ satisfies
\begin{equation*}
    \m E \| \widehat{\b x}_k - \b x\| \lesssim \norm{\vec{x}} \cdot \tilde{\kappa}(\mat{A})^2u. 
\end{equation*}
\end{informaltheorem}
A formal statement of this result is given in \cref{thm:accelkz}.
This result shows that ARK in floating-point arithmetic maintains its faster convergence over randomized Kaczmarz, as it always holds that $\tilde \kappa (\mat A) \sqrt n \le \tilde \kappa(\mat A)^2$. (Note that $\tilde \kappa (\mat A)\geq \sqrt n$ so the only way the two convergence rates are comparable is if $\mat A$ is perfectly conditioned.)
As with base Kaczmarz, the method may be stabilized using iterative refinement.
\begin{informaltheorem} [Stability of ARK with iterative refinement] \label{infthm:accelkz_ir}
    Assume that the invertible $\mat A$ satisfies $\tilde \kappa(\mat A)^2 u \ll 1$. Suppose that we perform iterative refinement $t = \log_2(\f 1 {\tilde{\kappa}(\mat{A}) u})$ times, where each refinement executes $\order(\tilde \kappa(\mat {A})\sqrt n)$ iterations of ARK. Then, with particular choices of algorithm parameters, this scheme attains forward stability:
    $$\m E \norm{\wb x_t - \b x } \lesssim \norm{\b x} \cdot \tilde \kappa(\mat A) u.$$
\end{informaltheorem}
A formal statement appears in \cref{thm:accelkz_ir}.
Taken together, the results of this paper show that, to achieve both rapid convergence and high accuracy, it is necessary to combine acceleration and iterative refinement.

\subsection{Organization} \label{sec:org}
The paper is organized as follows. \cref{sec:notation} discusses notation and defines what it precisely means for a method to be forward stable. In \cref{sec:kz}, we examine the stability of randomized Kaczmarz. \cref{sec:ir} proves a black-box result on iterative refinement that can be applied to any (possibly randomized) numerical method achieving a general accuracy condition, and the result is applied in \cref{sec:app_kz} to stabilize randomized Kaczmarz. In \cref{sec:rak}, we perform a similar analysis on accelerated Kaczmarz with and without iterative refinement. We verify the theoretical results on randomly generated test suites in \cref{sec:experiment}, and we conclude in \cref{sec:conclusion}.

\subsection{Notation} \label{sec:notation}

Throughout, we work with a square linear system 
\begin{equation*}
    \mat{A}\vec{x} = \vec{b} \quad \text{for invertible } \mat{A} \in \mathbb{R}^{n\times n} \text{ and }\vec{b} \in \mathbb{R}^n
\end{equation*}
The solution is denoted $\vec{x} \in \mathbb{R}^n$.
The Demmel condition number of $\mat{A}$ is $\tilde\kappa(\mat{A}) \coloneqq \norm{\mat{A}}_{\rm F} \norm{\mat{A}^{-1}}$.
Here, the $\norm{\cdot}$ denotes the vector $\ell_2$ norm and its induced operator norm, and $\norm{\cdot}_{\rm F}$ is the Frobenius norm.
The Demmel condition number is always within a $\sqrt{n}$ factor of the ordinary condition number $\kappa(\mat{A}) = \norm{\mat{A}}\norm{\mat{A}^{-1}}$:
\begin{equation} \label{eq:condition-equivalence}
    \kappa(\mat{A}) \le \tilde\kappa(\mat{A}) \le \sqrt{n}\,\kappa(\mat{A}).
\end{equation}

We use a modified version of Vinogradov notation where $a \lesssim b$ indicates that $a \le C(n)b$ for a fixed, universal, but unspecified \warn{polynomial} function $C(n) > 0$ of the matrix dimension $n$.
We write $a\asymp b$ when $a\lesssim b \lesssim a$.
We write $a \ll b$ to indicate $a \le b / C(n)$ for a \warn{sufficiently large} universal polynomial function $C(n) > 0$.

The numerically computed version of a quantity $f$ is denoted by $\widehat f$, and the error is $\err(f) \coloneqq \widehat f - f$.
A quantity is \textit{exactly represented} if $\hat{f} = f.$
The function $\fl(\cdot)$ rounds its input to the nearest floating-point number.
When applying the $\err(\cdot)$ notation to an expression, we assume that it is evaluated in the natural way.
For instance, $\err(\mat{A}\vec{z})$ denotes the floating-point error of the standard matrix--vector multiplication routine on inputs $\hat{\mat{A}}$ and $\hat{\vec{z}}$.
As a warning, note that $\err(\cdot)$ is typically not a mathematical function of its input.
For instance, $\err(x + y) = \fl(\hat{x} + \hat{y}) - (x+y)$ is not a function of the input $x+y$, as $\fl(\hat{x} + \hat{y})$ is not uniquely determined by the value of $x + y.$
We work in the standard model of floating-point arithmetic:
For any $x \in \m R$, $|\fl(x) - x| \le u |x|$ and, for exactly represented scalars $a, b \in \m R$, we have 
    $$ | \err(a \mathbin{\mathrm{op}} b) | \le | a \mathbin{\mathrm{op}} b |\, u \quad \text{for } \mathrm{op} \in \{+, -, \times, \div \}.$$

For the iterates of an algorithm, we use the following:
\begin{itemize}
    \item $\b x_0, \b x_1, \dots$ are the iterates computed in exact arithmetic
    \item $\widehat{\b x}_0, \widehat{\b x}_1, \dots$ are the iterates computed in floating-point arithmetic, i.e., for $k \ge 1$, $\widehat {\b x}_k$ is the floating-point output of one iteration on input $\widehat {\b x}_{k-1}.$
    \item For $k \ge 1,$ $\tilde {\b x}_k$ is the exact output of one iteration on input $\widehat {\b x}_{k-1}$
    \item For $k \ge 1$, the error is $\err(\b x_{k}) \coloneqq  \widehat {\b x}_k - \tilde {\b x}_k$. 
\end{itemize}
The relation between these quantities is summarized in \cref{fig:iterate_relation}.
Note that, for the iterates of an algorithm, we have overloaded the $\err(\cdot)$ function to report only the error incurred on \emph{that} iteration.

\begin{figure}[t]
    \centering
    \begin{tikzpicture}[>=stealth, scale=0.9, transform shape]
        \node (xk) at (0,1.3) {$\b x_k$};
        \node (xkhat) at (0,0) {$\widehat{\b x}_k$};
        \node (xktilde) at (0,-1.3) {$\tilde{\b x}_k$};

        \node (xkp1) at (5.2,1.3) {$\b x_{k+1}$};
        \node (xkp1hat) at (5.2,0) {$\widehat{\b x}_{k+1}$};
        \node (xkp1tilde) at (5.2,-1.3) {$\tilde{\b x}_{k+1}$};

        \draw[->, thick] (xktilde) -- (xkhat);
        \draw[->, thick] (xkp1tilde) -- node[right] {$\err(\hat{\vec{x}}_{k+1})$} (xkp1hat);

        \draw[->, thick] (xk) -- node[above, midway] {exact update} (xkp1);
        \draw[->, thick] (xkhat) -- node[below=6pt, midway] {exact update} (xkp1tilde);
        \draw[->, thick, dashed] (xkhat) -- node[above] {floating-point update} (xkp1hat);

        \draw[->, dashed, thick] (-2.0,1.3) -- (xk);
        \draw[->, dashed, thick] (xkp1) -- (7.2,1.3);
        \draw[->, dotted, thick] (-2.0,0) -- (xkhat);
        \draw[->, dotted, thick] (xkp1hat) -- (7.2,0);

        \node[anchor=east, font=\small] at (-2.2,1.3) {exact iterates};
        \node[anchor=east, font=\small] at (-2.2,0) {floating-point iterates};

        \node[font=\small] at (0,2.1) {$k$};
        \node[font=\small] at (5.2,2.1) {$k+1$};
    \end{tikzpicture}
    \caption{Relation of the exact Kaczmarz iterates $\b x_k$ and the floating-point iterates $\widehat{\b x}_k$. The exact update map takes $\b x_k$ to $\b x_{k+1}$, and the same exact update map takes $\widehat{\b x}_k$ to $\tilde{\b x}_{k+1}$. The floating-point update map takes $\widehat{\b x}_k$ to $\widehat{\b x}_{k+1}$, and the error incurred by this map is $\err(\b x_{k+1}) = \widehat{\b x}_{k+1} - \tilde{\b x}_{k+1}$.}
    \label{fig:iterate_relation}
\end{figure}
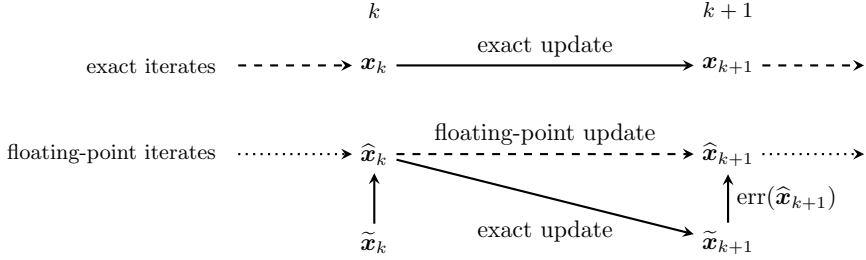

An algorithm used to solve a consistent linear system $\mat A \mat x = \b b$ is said to be \textit{forward stable} if it produces a floating-point output $\wb x$ satisfying
$$ \norm{\wb x - \b x } \lesssim \norm {\b x} \cdot \kappa(\mat A) u.$$ 
By \cref{eq:condition-equivalence}, the ordinary and Demmel condition numbers are equivalent, $\kappa(\mat{A}) \asymp \tilde\kappa(\mat{A})$, so a method is also forward stable if $\norm{\err(\vec{\wb x})} \lesssim \norm{\vec{x}} \cdot \tilde\kappa(\mat{A})u$. 
A method that is not forward stable will be called \emph{forward unstable}.
We use the following basic stability properties throughout \cite[sec.~3]{higham2002accuracy}:

\begin{theorem}[Basic stability properties] \label{thm:basic_stability}
    Assume $u \ll 1$.
    Then:
    \begin{enumerate}[label=(\alph*)] 
    \item For any vector $\b w$, we have 
    $$ \| \fl(\b w) - \b w \| \lesssim \| \b w \| u.$$    
    \item For exactly represented vectors $\b w$ and $\b z$ and exactly represented scalar $\gamma \in \m R$, we have $$ \| \err(\b z \pm \b w) \| \lesssim \| \b z \pm \b w \| u, \enspace \| \err(\gamma \b z) \| \lesssim |\gamma| \|\b z\| u, \enspace \norm {\err(\langle \b z,\b w\rangle)} \lesssim \norm { \b z} \norm{\b w} u.$$
    
    \item For an exactly represented matrix $\mat A$ and an exactly represented vector $\b v$, 
    $$ \| \err(\mat A\b v) \| \lesssim \| \mat A \|  \|\b v\| u.$$
    \end{enumerate}
\end{theorem}

\subsection{Related work}
Given the amount of work on the (randomized) Kaczmarz iteration and stochastic gradient descent, the scarcity of floating-point analyses of these methods in the literature is somewhat surprising. Among deterministic methods, a floating-point analysis has been developed for conjugate gradient \cite{greenbaum1989behavior,musco2018stability} and gradient descent \cite{STL22}, while \cite{spielman2014nearly} perform a finite-precision analysis of Chebyshev iteration. In the context of randomized methods, \cite{lee2013efficient} provides a stability analysis of an accelerated variant of coordinate descent on positive semidefinite systems, while \cite{XMHK23,XMH25} consider gradient descent under a stochastic rounding model of computation.


The literature on randomized Kaczmarz is vast, and it is beyond our scope here to summarize it. Of note, there has been a significant effort to understand how randomized Kaczmarz behaves on noisy linear systems, and to design variants of the algorithm that are robust in various ways to noise (see, e.g., \cite{Ste23,HMR23,BBD+24,Nee10,zfn13,EGW26}).
These works differ from ours in that they assume noise in the \emph{data}; for us, noise comes from rounding errors that occur during computation. Relatedly, \cite{derezinski2024solving} provided a stability analysis of randomized Kaczmarz in the model where the projection step is computed approximately (not necessarily due to rounding errors) but all other steps are performed in exact arithmetic. Still, this does not provide a clear picture of the stability of Kaczmarz in floating-point arithmetic.

More recently, there has been growing interest in accelerated Kaczmarz methods, which aim to improve convergence rates while preserving the low per-iteration cost of randomized Kaczmarz. Several approaches draw inspiration from first-order optimization, incorporating momentum or variance reduction techniques into the iteration \cite{LW13, gower2018accelerated}. In particular, \cite{derezinski2025fine} extended the partial stability analysis under inexact projections to an accelerated variant of Kaczmarz, and \cite{derezinski2025randomized} combined acceleration with other techniques such as blocking and regularization to achieve convergence rates beyond those obtained by Krylov subspace methods.


Iterative refinement is a classic technique in numerical computations \cite[ch.~12]{higham2002accuracy}; it was first analyzed in Wilkinson's classic monograph \emph{Rounding Errors in Algebraic Processes} \cite{Wil23}.
Interest in iterative refinement has been heightened in recent years due to increased interest in low-precision and mixed precision computing \cite{carson2018accelerating,HM22}.
Traditionally, iterative refinement is used to improve solutions to a linear system of equations using a matrix factorization of $\mat{A}$ computed in a lower numerical precision (i.e., a solution is demanded in double precision but $\mat{A}$ is factored in single-precision).
More recently, there has been interest in using iterative refinement in a uniform numerical precision to improve the attainable accuracy of an iterative method \cite{EGN25,EMN26}.

Our paper fits into a growing effort to understand the numerical stability of algorithms from the field of randomized linear algebra \cite{EMN26,Epp25a,BNP25a,PN25,MNTW23,Nak20,BGS25,chenakkod2026well}.
In broad strokes, the conclusions of this literature have been that, when implemented appropriately, randomized methods can reach levels of accuracy and stability comparable to standard direct methods.


\section{Randomized Kaczmarz} \label{sec:kz}
In this section, we prove \cref{thm:base_kz}, which shows that the randomized Kaczmarz algorithm converges geometrically until the error stagnates at a level not exceeding $\tilde{\kappa}(\mat{A})^2 u$.
Since this convergence horizon exceeds the level $\tilde{\kappa}(\mat{A}) u$, this analysis is consistent with the numerical observation in \cref{fig:exp_1e4} that randomized Kaczmarz is \emph{not} forward stable. 
We show in \cref{sec:ir_app_kz} that this instability can be cured by combining randomized Kaczmarz with iterative refinement.



To prove \cref{thm:base_kz}, we begin by quantifying the floating-point error $\err(\vec{x}_k)$ incurred during a single iteration of the algorithm.
Throughout, we use the notation for the iterates defined in \cref{sec:notation}; see also \cref{fig:iterate_relation} for a visual summary of the notation.

\begin{lemma}[Randomized Kaczmarz: Single iteration] \label{lem:base_kz_error}
    The error $\err(\b x_{k+1})$ incurred during the $(k+1)$st iteration of randomized Kaczmarz satisfies
    $$ \| \err (\b x_{k+1}) \| \lesssim \left ( \|\b x \| + \| \widehat{\b x}_k \| \right) u. $$
\end{lemma}

\begin{proof}
    A straightforward application of the basic stability results (\cref{thm:basic_stability}) yields the bound
    \begin{equation*}
        \| \err (\b x_{k+1}) \| \lesssim \frac{|b_{r(k)}| + \norm{\vec{a}_{r(k)}}\norm{\widehat{\vec{x}}_k}}{\norm{\vec{a}_{r(k)}}} \cdot u.
    \end{equation*}
    Since $\mat{A}\vec{x} = \vec{b}$, it holds that $|b_{r(k)}| = |\langle \vec{a}_{r(k)},\vec{x}\rangle| \le \norm{\vec{a}_{r(k)}}\norm{\vec{x}}$.
    The stated bound follows.
\end{proof}


We now combine the local floating-point error bound from \cref{lem:base_kz_error} with the standard analysis of randomized Kaczmarz to analyze the algorithm in finite-precision arithmetic, proving \cref{thm:base_kz}.


\begin{proof}[Proof of \cref{thm:base_kz}]
By definition $\tilde{\vec{x}}_{k+1}$ is the result of applying one step of randomized Kaczmarz to $\widehat{\vec{x}}_k$ in exact arithmetic
\begin{equation*} \m E \big[\| \tilde {\b x}_{k+1} - \b x\| \mid \hat{\vec{x}}_k \big] \le \left(\m E \big[\| \tilde {\b x}_{k+1} - \b x\|^2 \mid \hat{\vec{x}}_k \big] \right)^{1/2} \le \sqrt {1 - {\tilde \kappa}(\mat{A})^{-2}} \cdot \| \widehat {\b x}_{k} - \b x\|. 
\end{equation*}
Here, $\tilde{\kappa}(\mat{A})$ is the Demmel condition number \cref{eq:demmel}.
Combining this result with \cref{lem:base_kz_error}, we obtain the error bound
\begin{equation} \label{eq:rk_one_step_bound}
\begin{split}
\m E \| \widehat{\b x}_{k+1} - \b x\|
&\le 
\m E \| \widehat{\b x}_{k+1} - \tilde {\b x}_{k+1} \| + \m E \| \tilde {\b x}_{k+1} - \b x \|
\\& \le \m E \left[C \left( \| \b x \| + \| \widehat {\b x}_k \| \right) u\right] + \sqrt{1 - {\tilde \kappa}(\mat{A})^{-2}} \cdot \m E \| \widehat {\b x}_{k} - \b x\|.
\end{split}
\end{equation}
Here, $C\lesssim1$ is the prefactor suppressed in \cref{lem:base_kz_error}.


Define $e_k \coloneqq \m E\| \widehat{\b x}_{k} - \b x \|,$ and let $r = \sqrt{1 - {\tilde \kappa}(\mat{A})^{-2}}$.
Then the one-step bound \cref{eq:rk_one_step_bound} may be equivalently written
$$ e_{k+1} \le C( \| \b x \| + \m E \| \widehat {\b x}_k \| )u + r e_{k}.$$
Inductively assume a bound 
$ \m E\| \widehat{\b x}_k \| \le  \| \b x\|,$ so the above inequality implies
$$ e_{k+1} \le 2C \| \b x \|u + r e_k.$$
The solution to this inhomogeneous linear differential inequality gives
\begin{equation*}
\m E \| \widehat{\b x}_k - \b x\| \le \left(\| \widehat {\b x}_0 - \b x \| - \f {2C \| \b x\| u} {1-r}\right) r^k + \f {2C \| \b x\| u} {1-r}.
\end{equation*}
Since $r < 1 - \tilde\kappa(\mat{A})^{-2}/2$, we obtain the bound
\begin{equation*}
\m E \| \widehat{\b x}_k - \b x\| \le \left(1 - \frac{1}{2\tilde{\kappa}(\mat{A})^{2}}\right)^k\| \widehat {\b x}_0 - \b x \| + 4C  \| \b x\| \cdot \tilde\kappa(\mat{A})^2u.
\end{equation*}
By the assumption $\tilde{\kappa}(\mat{A})^2u \ll 1$, we confirm the inductive assumption $\m E \norm{\vec{\hat{x}}_k} \le \norm{\vec{x}}$ holds.
The desired bound has been proven.
\end{proof} 

\begin{remark}[Is randomness to blame?] 
    One might initially suspect that the failure of the randomized Kaczmarz method to be forward stable arises from the randomness itself---that it is the fact that the rows are chosen randomly that introduces noise which compounds over iterations.
    But as the analysis lays bare, this is not the case: 
    The lack of forward stability arises because the algorithm incurs a small error at each step, and the method converges at a sufficiently slow rate that these errors accumulate to a level above that of a forward stable method.
    The analysis suggests the same error propagation would arise with a deterministic selection of rows.
    Indeed, the stability properties of deterministic Kaczmarz algorithms could even be worse than the randomized method, as bad row orderings can lead to much slower convergence than randomized Kaczmarz.
\end{remark}

\section{Iterative refinement and its application to randomized Kaczmarz} \label{sec:ir_app_kz}

This section develops general stability results for iterative refinement (\cref{sec:ir}), showing that any solver which reduces the error by a constant factor at every iteration produces a forward stable solution after enough steps of iterative refinement.
In \cref{sec:app_kz}, we use this result to prove randomized Kaczmarz with iterative refinement is forward stable.

\subsection{Iterative refinement} \label{sec:ir}
We prove that iterative refinement upgrades an algorithm that contracts the error at every step to one that is forward stable.
Our analysis may be seen as a variant of the results of \cite[sec.~3]{JW77} that has been generalized to treat randomized algorithms.
To prove our results in a bit more generality, we work with a rectangular matrix $\mat{A} \in \m R^{m\times n}$ for this subsection.

\begin{theorem}[Iterative refinement with expectation] \label{thm:refinement_exp}
    Let $\mat{A} \in \m R^{m \times n}$, and assume $\mat{A}$ has full column-rank.
    Consider a floating-point subroutine $\Call{Solver}{}$ satisfying
    \begin{equation} \label{eq:ir_condition}
        \m E \norm{ \Call{Solver}{\vec{b}} - \mat{A}^+\vec{b}} \le q \cdot \norm{\mat{A}^+\vec{b}}
    \end{equation}
    for some $q$ and all inputs $\b b$, where $\mat{A}^+$ denotes the Moore-Penrose pseudoinverse of $\mat A$. There exists an absolute polynomial $C = C(m, n)$, such that if $(1 + Cu) q < 1/2$ and $C\tilde \kappa(\mat A) u < 3/4,$ and if $\b b \in \mathrm{Range}(\mat A)$,
    then the iterative refinement scheme 
    \begin{align*}
        \vec{x}_0 &\coloneqq \Call{Solver}{\vec{b}}, \\
        \vec{x}_{i+1} &\coloneqq \vec{x}_i + \Call{Solver}{\vec{b} - \mat{A}\vec{x}_i}, \text{ for } i = 0, 1, \ldots, t-1
    \end{align*}
    produces a floating-point solution $\wb x_t$ satisfying
    $$\m E \norm{\wb x_t - \mat{A}^+ \b b } \le (C\tilde{\kappa}(\mat{A}) u + (1 + Cu)^t q^{t+1}) \cdot \norm{\mat{A}^+ \b b}.$$
    The Demmel condition number for a matrix $\mat{A}$ of full column rank is defined as $\tilde{\kappa}(\mat{A}) = \norm{\mat{A}}_{\rm F} \norm{\mat{A}^+}$.
    Consequently, forward stability is achieved in expectation in
    $t =  \log_2(\f{1}{\tilde{\kappa}(\mat{A}) u})$ refinements, in the sense that 
     \begin{align*}\m E \norm{\wb x_t - \mat{A}^+ \b b } \lesssim \tilde{\kappa}(\mat{A}) u \cdot \norm{\mat{A}^+ \b b}.\end{align*}
\end{theorem}

\begin{proof}
    Let $ e_i = \m E \norm{\wb x_i - \mat {A}^+ \b b},$ let $\b r_i = \b b - \mat{A} \wb x_i,$ and let $\wb r_i$ be the result of computing $\b b - \mat{A} \wb x_i$ in floating point. Inductively assume that $e_i \le \norm{\mat A^+ \b b }.$ Then,
    \begin{align*}
    \m E [\err(\b r_i)] \lesssim (\norm{\b b} + \norm{\mat A}\m E [\norm{ \wb x_i}])u \lesssim \norm{\mat A}\norm{\mat A^+ \b b} u.
    \end{align*}
    In the second equation, we used the bound $\norm{\b b} = \norm{\mat{A}\mat{A}^+\vec{b}}\le \norm{\mat{A}} \norm{\mat{A}^+\vec{b}}$ and the triangle inequality $\m E\norm{\wb x_i} \le \norm{\mat{A}^+\vec{b}} + e_i \le 2\norm{\mat{A}^+\vec{b}}$.
    Hence, for some polynomials $c_1 = c_1(m, n)$ and $c_2 = c_2(m, n)$,
    \begin{align*}
    \m E \norm{ \wb{x}_{i+1} - \mat A^+ \b b } &\le 
    (1 + c_1u) q \cdot \m E \norm{ \wb{x}_i - \mat A^+ \b b } + (1 + c_1u) \m E \norm{ \mat A^{+} \err(\b r_i) } + c_1 u \norm{\mat {A}^+ \b b}
    \\&\le (1 + c_1u) q \cdot \m E \norm{ \wb{x}_i - \mat A^+ \b b } + c_2 \tilde{\kappa}(\mat{A}) u \cdot \norm{\mat A^+ \b b}.
    \end{align*}
    Thus, by solving the inhomogeneous linear differential inequality, we obtain
     \begin{align*} 
     e_{i} &\le \left(e_0 - \f {c_2 \tilde{\kappa}(\mat{A}) u \cdot \norm{\mat A^+ \b b}} {1- (1 + c_1u) q}\right) (1 + c_1u)^i q^i + \f {c_2 \tilde \kappa(\mat A) u} {1-(1 + c_1u) q} \cdot \norm{\mat A^+ \b b} 
     \\&\le ((1 + Cu)^i q^{i+1} + C \tilde \kappa (\mat A) u) \cdot \norm{\mat A^+ \b b},
     \end{align*}
     for some polynomial $C = C(m, n)$. Note that the inductive hypothesis $e_i \le \norm{\mat A^+ \b b}$ holds if $C \tilde \kappa(\mat A) u < 3/4.$
\end{proof}

We conclude this section with two remarks on this result.
\begin{remark} [Number of refinements] \label{rem:t}
If the system is ill-conditioned, then $\tilde{\kappa}(\mat{A}) u $ is large, and thus $t = \log_2\big(\tfrac{1}{\tilde{\kappa}(\mat{A}) u}\big)$ suggests that \textbf{less} refinements are needed, which may be counter-intuitive. The explanation for this is that when $\tilde{\kappa}(\mat{A}) u$ is large, the forward stability requirement $\m E \norm{\wb x_t - \mat{A}^+ \b b } \lesssim \tilde{\kappa}(\mat{A}) u \cdot \norm{\mat A^+ \b b}$ is less restrictive, so fewer refinements are needed. The difficulty in solving the ill-conditioned system appears when attempting to satisfy the condition \eqref{eq:ir_condition}.

In single precision where $u \approx 10^{-8}$, in the worst case scenario where $\tilde{\kappa}(\mat{A}) \approx 1$, we require at most $\log_2(\tfrac 1 u) \approx 27$ refinements, while double precision with $u \approx 10^{-16}$ requires at most $\log_2(\tfrac 1 u) \approx 54$ refinements.
If $q \ll 1$, the number of iterations can be \emph{much} smaller.
Indeed, under the mild condition $2\, \log \big(\tfrac 2 {1 + Cu} \big) > 1$, the maximum number of iterative refinement steps is at most 
\begin{equation} \label{eq:kz_ir_iterations}
t =\frac{ \log\big(\tfrac{1}{\tilde{\kappa}(\mat{A}) u}\big)}{ \log \big(\tfrac{1}{(1 + Cu) q}\big) }\le (1 + 2\log(1 + Cu))\frac{\log\big(\tfrac{1}{\tilde{\kappa}(\mat{A}) u}\big)}{ \log \big(\tfrac{1}{ q}\big)} \approx \frac{\log\big(\tfrac{1}{\tilde{\kappa}(\mat{A}) u}\big)}{ \log \big(\tfrac{1}{ q}\big)}.
\end{equation}
Here, $1 + 2 \log(1 + Cu)$ is close to 1, unless the system is very large.
\end{remark}

\begin{remark} [Expectation]
 Notice that the hypothesis \eqref{eq:ir_condition} places an assumption on the expected error after one step of the procedure.
 This is necessary for our us as we analyze randomized methods.
 One can also apply \cref{thm:refinement_exp} to deterministic methods, in which case the expectation can be ignored.
\end{remark}

\subsection{Application to randomized Kaczmarz} \label{sec:app_kz}
We saw in \cref{thm:base_kz} that randomized Kaczmarz is not forward stable, with convergence horizon proportional to $\norm{\b x} \cdot \tilde \kappa(\mat A)^2 u.$
In this section, we prove \cref{thm:base_kz_ir}, which shows that, provided $\tilde{\kappa}(\mat A)^2 u \ll 1$, randomized Kaczmarz with iterative refinement achieves error $\norm{\b x} \cdot \tilde{\kappa}(\mat A) u$, the characteristic error for a forward stable method.
Let us emphasize that this result \cref{thm:base_kz_ir} is only a \emph{conditional} stability result, as we assume $\tilde{\kappa}(\mat A)^2 u \ll 1$.

Before proving \cref{thm:base_kz_ir}, a few words about implementation of randomized Kaczmarz with iterative refinement are in order.
To implement the method with minimal storage, one only need store two arrays beyond the storage needed for the inputs $\mat{A}$ and $\vec{b}$.
The first vector $\vec{x}$ stores the current solution, and the second vector $\vec{e}$ stores the current approximate solution to the residual equation $\mat{A}\vec{e} = \vec{r} \coloneqq \vec{b} - \mat{A}\vec{x}$.
The residual vector $\vec{r}$ need not be stored, as individual entries $r_i = b_i - \langle \vec{a}_i, \vec{x}\rangle$ can be generated on the fly.
With this implementation, the total cost of a single step of randomized Kaczmarz is $8n$ floating point operations, as compared to the $6n$ floating point operations for randomized Kaczmarz without iterative refinement.

Now, we prove \cref{thm:base_kz_ir} by combining \cref{thm:base_kz} with \cref{thm:refinement_exp}.
\begin{proof} [Proof of \cref{thm:base_kz_ir}]
    Let $k_i$ be the number of iterations of randomized Kaczmarz in the $i$th refinement, and let $C_1(n)$ be a polynomial so that~\cref{thm:base_kz} gives,
    \begin{align*}
    \m E \| \widehat{\b x}_{1} - \b x\| 
    \le \left[  \left(1 - \f {1} {2\tilde{\kappa}(\mat{A})^{2}}\right) ^{k_1} + C_1 \tilde{\kappa}(\mat{A})^2 u \right] \norm{\b x}.
    \end{align*} 

    Then, during the $i$th refinement, $ \left[ \left(1 - \f {1} {2\tilde{\kappa}(\mat{A})^{2}}\right) ^{k_i} + C_1 \tilde{\kappa}(\mat{A})^2 u \right]$ is the $q$ defined in~\cref{thm:refinement_exp}. The theorem will immediately give the desired result, so it suffices to show that $(1 + C_2u) q < 1/2$ and $C_2 \tilde \kappa(\mat A) u < 3/4$, where we set $C_2$ to be the polynomial $C$ defined in~\cref{thm:refinement_exp}. If we choose $k_i$ to satisfy 
    $$ k_i > \f {\log \left(\f  1 {2(1 + C_2u)} - C_1 \tilde \kappa(\mat{A})^2 u \right)} { \log \left(1 - \f {1}{2\tilde \kappa(\mat{A})^{2}} \right) } = \order(\tilde \kappa(\mat A)^2),$$ then
    \begin{align*} 
    (1 + C_2u) q 
    < (1 + C_2u) \left[ \left(\f  1 {2(1 + C_2u)} - C_1\tilde \kappa(\mat{A})^2 u \right) + C_1\tilde \kappa(\mat{A})^2 u \right]
    = \f 1 2.
    \end{align*}
    The second requirement, $C_2 \tilde \kappa(\mat A) u < 3/4$, is implied by the assumption $\tilde \kappa(\mat{A})^2 u \ll 1.$
\end{proof}

\section{Accelerated randomized Kaczmarz} \label{sec:rak}
In the previous section, we observed that iterative refinement can help stabilize randomized Kaczmarz.
However, as per \eqref{eq:kz_ir_iterations}, we still need a large number of iterations to do so, especially when the system is ill-conditioned.
Thus, in search of a method that is both fast and stable, we turn to analyzing accelerated randomized Kaczmarz (ARK).

Because the instability of randomized Kaczmarz arises from slow convergence and the resulting accumulation of floating-point errors, one might expect the more rapidly converging accelerated Kaczmarz method to be forward stable even without iterative refinement.
However, as we demonstrate empirically and support with rounding-error analysis, this is not the case.

We study a version of ARK that uses two auxiliary iterate sequences, as proposed by \cite{gower2018accelerated}, and an additional regularization term as suggested by \cite{derezinski2025randomized}. This acceleration scheme has received significant attention recently for its strong convergence guarantees \cite{derezinski2025fine,derezinski2025randomized} and empirical performance \cite{rathore2024have,rathore2026turbocharging}.
The method requires three tunable parameters $\tilde \mu$, $\tilde \nu$, and $\lambda$; we discuss how to set these parameters below.
Using these values, we define auxiliary parameters $\beta \coloneqq 1 - \sqrt{\tilde \mu / \tilde \nu}, \gamma \coloneqq 1 / \sqrt{\tilde \mu \tilde \nu},$ and $\alpha \coloneqq 1 / (1 + \gamma \tilde \nu)$ and initialize $\b v_0 \coloneqq  \b x_0$ and $\b y_0 \coloneqq  \b x_0$.
The ARK method is based on the following update rule:
\begin{equation}
\begin{cases}
    \b x_k = \alpha \b v_k + (1 - \alpha) \b y_k \\
    \text{Pick row index } r(k) \text { from } [n] \text{ with probability } \propto \norm{\b a_{r(k)}}^2 + \lambda \\ 
    \b w_k = \f {\langle \b a_{r(k)}, \b x_k \rangle - b_{r(k)}} { \| \b a_{r(k)} \| ^2 + \lambda} \b a_{r(k)} \\
    \b y_{k+1} = \b x_k - \widehat{\b w}_k \\
    \b v_{k+1} = \beta \b v_k + (1 - \beta) \b x_k - \gamma \widehat {\b w}_k.
\end{cases} \label{eq:ARK}
\end{equation}
Note that we choose row $r(k)$ with probability proportional to the \emph{regularized} probabilities $\norm{\b a_{r(k)}}^2 + \lambda$.

\subsection{Choice of acceleration parameters}

This subsection will discuss how to choose the tunable parameters $\tilde \mu$, $\tilde \nu$, and $\lambda$.
First, we will fix the regularization $\lambda$ and discuss how to pick the acceleration parameters $\tilde \mu$ and $\tilde \nu$.
Next, we discuss the choice of the regularization $\lambda$.

For now, assume $\lambda \ge 0$ is fixed.
To reason about the parameters $\tilde\mu$ and $\tilde\nu$, we introduce the regularized projection matrix
$\b P_{\lambda, i} = \f {\b a_{i} \b a_{i}^\top }{\norm{\b a_{i}}^2 + \lambda}$ associated with row index $i$, and let $\overline {\b P}_{\lambda} \coloneqq \m E[\b P_{\lambda, i}]$ denote its expectation for a random index $i$ selected according to the regularized randomized Kaczmarz distribution.
The ideal ARK acceleration parameters take the form
\begin{equation}\mu \coloneqq  \lambda_{\rm min}^{+}( \overline{\b P}_{\lambda}), \enspace \nu \coloneqq  \lambda_{\rm max}(\m E[( \overline{\b P}_\lambda^{\dagger/2} \b P_{\lambda, r(k)} \overline {\b P}_\lambda^{\dagger/2})^2] ). \label{eq:mu_nu}
\end{equation}
Naturally, these parameters are difficult to estimate.

Fortunately, one can prove convergence guarantees for any choice of parameters $\tilde \mu \le \mu$ and $\tilde \nu \ge \nu$.
In particular, given such parameters, the ARK method satisfies the following recursive bound \cite[Lemma 2.3]{derezinski2025randomized} in exact arithmetic:
\begin{equation}
    \m E \left [ \| {\b v}_{k+1} - \b x\|^2_{\overline {\b P}_\lambda^\dagger} + \f 1 {\tilde \mu} \| {\b y}_{k+1} - \b x \|^2 \right ] \le \left(1 - \f 1 2 \sqrt{ \f {\tilde \mu}{\tilde \nu}} \right)  \left( \| {\b v}_k - \b x\|^2_{\overline {\b P}_\lambda^\dagger} + \f 1 {\tilde \mu} \| {\b y}_k - \b x \|^2\right).
\end{equation}
Thus, the ratio $\tilde \mu / \tilde \nu$ determines the rate of convergence, so it is optimal to choose $\tilde \mu$ and $\tilde \nu$ to be as close as possible to the ideal parameters $\mu$ and $\nu$. The authors in \cite{derezinski2025randomized} prove bounds on $\mu$ and $\nu$ when $\mat{A}$ is preprocessed with a randomized Hadamard transform that ensures that $\mat{A}$ satisfies an incoherence property.
The following lemma provides an alternate bound that holds in general and does not require preprocessing. The proof is deferred to~\cref{sec:mu_nu}.

\begin{lemma}[Acceleration parameters] \label{lem:mu_nu}
The optimal parameters $\mu$ and $\nu$ governing the rate of acceleration satisfy the following:
    \begin{align*}
        \mu &= \lambda_{\rm min}^+(\overline{\b P}_\lambda) = \f {\sigma_{\rm min}^2(\mat A)} { \norm{\mat A}_{\rm F}^2 + n \lambda}, \\ 
        \nu &= \lambda_{\rm max}(\m E[( \overline{\b P}_\lambda^{\dagger/2} \b P_{\lambda, i} \overline {\b P}_\lambda^{\dagger/2})^2]) \le \f {\norm{\mat A}_{\rm F}^2 + n \lambda} { \min_i \norm{\b a_i}^2 + \lambda }.
    \end{align*}
\end{lemma}

\cref{lem:mu_nu} shows that we may pick $\tilde \mu \coloneqq \f {\sigma^2_{\rm min}(\mat A)} { \norm{\mat A}_{\rm F}^2 + n \lambda}$ and $\tilde \nu \coloneqq \f {\norm{\mat A}_{\rm F}^2 + n \lambda}{\min_i \norm{\b a_i}^2 + \lambda}$.
Of course, these values are still difficult to compute efficiently.
For ease of analysis, we shall assume these near-ideal values are available to us. 
Under this assumption, the ratio governing the acceleration is 
$$ \f {\tilde \mu } { \tilde \nu} = \f { \min_i \norm{\b a_i}^2 + \lambda }{ \norm{\mat A}_{\rm F}^2 + n \lambda}  \cdot \f {\sigma^2_{\rm min}(\mat A)} { \norm{\mat A}_{\rm F}^2 + n \lambda}.$$

Now, we turn to selection of the regularization $\lambda$.
We consider two settings:
\begin{enumerate} [label=(\arabic*)]
\item \label{op:1} \textbf{No pre-processing, use regularization.} 
For any matrix $\mat{A}$, we can set $\lambda = \norm{\mat A}_{\rm F}^2 / n$ to obtain a controlled rate of convergence:
$$ \f {\tilde \mu } { \tilde \nu} \ge \f 1 {4n} \f {\sigma^2_{\rm min}(\mat A)}{\norm{\mat A}_{\rm F}^2}  = \f {\tilde \kappa(\mat A)^{-2}}{4n} \implies 1 - \f 1 2 \sqrt \f {\tilde \mu } { \tilde \nu} \le 1 - \f {1}{4 \tilde \kappa(\mat A)\sqrt n }.$$

\item \label{op:2} \textbf{Pre-process, skip regularization.} If we would like to avoid regularization (i.e., set $\lambda = 0$), we can do so by scaling the matrix $\mat{A}$ and right-hand side $\vec{b}$ so that the rows of $\mat{A}$ have equal or nearly equal norms.
In particular, if $4 \cdot \min_i \norm {\b a_i}^2 \ge \max_i\norm {\b a_i}^2,$ then we again obtain a controlled rate of convergence:
$$ \f {\tilde \mu } { \tilde \nu} \ge \f 1 {4n} \f {\sigma^2_{\rm min}(\mat A)}{\norm{\mat A}_{\rm F}^2}  = \f {\tilde \kappa(\mat A)^{-2}}{4n} \implies 1 - \f 1 2 \sqrt \f {\tilde \mu } { \tilde \nu} \le 1 - \f {1}{4 \tilde \kappa(\mat A)\sqrt {n} }.$$
\end{enumerate}
The conclusion is that, with either option \ref{op:1} or \ref{op:2}, ARK takes $\order(\tilde{\kappa}(\mat{A}) \sqrt{n})$ iterations to drive the error down by a constant factor.
The Demmel condition number satisfies the bound $\tilde{\kappa}(\mat{A}) \ge \sqrt{n}$, so the $\order(\tilde{\kappa}(\mat{A}) \sqrt{n})$ iteration cost of ARK never exceeds the $\order(\tilde{\kappa}(\mat{A})^2)$ iteration cost of standard randomized Kaczmarz.

\subsection{Finite-precision analysis}

We now present the formalized version of \cref{infthm:accelkz}, a general stability result for ARK.
\begin{theorem} [Suboptimal stability of ARK] \label{thm:accelkz}
    Let $\mat A \in \m R^{n \times n}$ be an invertible matrix, let $\b b \in \m R^{n}$ be a vector, and let $\b x \coloneqq  \mat {A}^{-1} \b b.$ Let $\widehat{\b x}_{k}$ be the output of $k$ iterations of regularized accelerated Kaczmarz with initialization $\b v_0 = \b y_0 = \b 0,$ executed in floating-point arithmetic with precision $u$.
     
    Furthermore, pick $\tilde \mu = \f {\sigma^2_{\rm min}(\mat A)} { \norm{\mat A}_{\rm F}^2 + n \lambda}$ and $\tilde \nu = \f {\norm{\mat A}_{\rm F}^2 + n \lambda}{\min_i \norm{\b a_i}^2 + \lambda},$ and assume that one of options \ref{op:1} and \ref{op:2} has been chosen.
     Finally assume that $\tilde \kappa(\mat A)^2 u \ll 1.$ Then it holds that
     $$ \m E \|\widehat{\b x}_k - \b x\| \le \left(1 - \f 1 {4 \tilde \kappa(\mat A) \sqrt{n}} \right)^{k/4}c \| \b x \| +  C\| \b x \| \cdot \tilde \kappa(\mat A)^2  u $$
     for some constant $c$ (independent of $n$) and some polynomial $C = C(n).$
     In particular, if we run for $k = \order(\tilde{\kappa}(\mat{A}) \sqrt n \log(1/u))$ iterations, then the output satisfies
\begin{equation*}
    \m E \| \widehat{\b x}_k - \b x\| \lesssim \norm{\vec{x}} \cdot \tilde{\kappa}(\mat{A})^2u.
\end{equation*}
\end{theorem}
Again, similarly to \cref{thm:base_kz}, the expected forward error decays up to a convergence horizon proportional to $\norm{\b x} \cdot \tilde \kappa(\mat A)^2 u.$ Thus, ARK is \textit{not} forward stable. As before, iterative refinement helps reduce the convergence horizon term. The following result is the formalized version of \cref{infthm:accelkz_ir}.
\begin{theorem} [Stability of ARK with iterative refinement] \label{thm:accelkz_ir}
    Let $\mat A \in \m R^{n \times n}$ be invertible, let $\b b \in \m R^{n}$ be a vector, and let $x = \mat {A}^{-1} \b b.$
    
     Pick $\tilde \mu = \f {\sigma^2_{\rm min}(\mat A)} { \norm{\mat A}_{\rm F}^2 + n \lambda}$ and $\tilde \nu = \f {\norm{\mat A}_{\rm F}^2 + n \lambda}{\min_i \norm{\b a_i}^2 + \lambda},$ and assume that one of options \ref{op:1} and \ref{op:2} has been chosen.
    Suppose that we perform iterative refinement $t = \log_2\left(\f 1 {\tilde \kappa(\mat A) u} \right)$ times, where each refinement executes $\order(\tilde \kappa(\mat {A})\sqrt n)$ iterations of ARK. If $\tilde \kappa(\mat{A})^2 u \ll 1$, then this scheme attains forward stability:
    $$\m E \norm{\wb x_t - \b x } \lesssim \norm{\b x} \cdot \tilde{\kappa}(\mat{A}) u.$$
\end{theorem}

Observe that $\kappa(\mat{A})^2u \ll 1$ is needed in this theorem because outside of this regime, we cannot expect ARK itself to make much progress. Inside the regime, we obtain forward stability, with error proportional to the desired $\norm{\b x} \tilde \kappa(\mat A) u.$ Here, we require
$$ \order(\tilde \kappa(\mat {A})\sqrt n) \cdot \left(\log_2\left( \tfrac 1 {\tilde{\kappa}(\mat{A}) u} \right) + 1 \right)$$
ARK iterations to do so. In comparison, ARK without iterative refinement needs 
$$ 4 \f {\log \left(\tilde \kappa(\mat{A})^2 u \right)} { \log \left(1 - \f 1 {4 \tilde \kappa(\mat{A})\sqrt n} \right) } = \order( \tilde \kappa(\mat A) \sqrt n) \cdot \log \left( \f 1 { \tilde \kappa(\mat A)^2 u} \right)$$
iterations to achieve its convergence horizon of $\norm{\b x} \tilde \kappa(\mat A)^2 u.$ Thus, the iterative refinement scheme requires roughly the same order of magnitude of ARK iterations to yield a significantly better result, despite the fact that the two methods are precisely the same in exact arithmetic. See \cref{lem:ir_accelkz_exact} for details.

Recall that randomized Kaczmarz with iterative refinement requires $\order(\tilde \kappa(\mat A)^2)$ Kaczmarz iterations (see \eqref{eq:kz_ir_iterations}), significantly more than accelerated Kaczmarz. Furthermore, as an additional comparison between ARK with refinement and randomized Kaczmarz with refinement, we will see experimentally that the former appears to achieve forward stability in some cases even when $\tilde \kappa(\mat A)^2 u > 1$; see \cref{fig:testsuite1e5}. The latter method did not replicate the same stability.

To prove \cref{thm:accelkz}, we begin by proving a result on the floating-point error induced by computing 
\begin{align*}
{\b w}_k &= \f {\langle \b a, {\b x}_k \rangle - b} { \| \b a \| ^2 + \lambda} \b a
\end{align*}
in finite precision.
Here, we have set $\b a\coloneqq  \b a_{r(k)}$ to be the row that we iterate on, and we set $b \coloneqq  b_{r(k)}$.
In addition, recall our notational convention. We use $\hspace{0.5em} \widehat{}\hspace{0.5em}$ for the floating-point outputs of the algorithm, e.g. the algorithm will produce in floating-point arithmetic the iterates $\wb x_0, \wb w_0, \widehat{\b y}_1, \widehat{\b v}_1, \dots $. On the other hand, we use $\hspace{0.5em}\tilde{}\hspace{0.5em}$ for the exact outputs on the previous floating-point iterates. That is, given the current floating-point iterates $\widehat{\b y}_k, \widehat{\b v}_k$, the next iterates are computed in exact arithmetic and denoted by $\tilde {\b x}_k, \tb w_k, \tilde {\b y}_{k+1}, \tilde {\b v}_{k+1}.$

\begin{lemma} \label{lem:wk}
Computing $\b w_k = \f {\langle \b a, \b x_k \rangle - b} { \| \b a \| ^2 + \lambda} \b a$ in finite precision yields an error of 
$$ \| \wb w_k - \tb w_k \| \lesssim (\| \b x \| + \| \wb x_k \|) u.$$
\end{lemma}
\begin{proof} 
Straightforwardly applying the basic stability results in \cref{thm:basic_stability} yields
\begin{align*} \| \wb w_k - \tb w_k \| \lesssim \f {  \| \b a \| |b| +  \| \b a \|^2 \| \tb x_k \|   } {\|\b a \|^2 + \lambda} u \le \left(\f {\norm{b} + \norm{\tb x_k}}{\norm{\b a}}\right) u \le (\norm{\b x_k} + \norm{\wb x}) u.\end{align*}
\end{proof}

\cref{lem:wk} can be seen as a sort of analogue to \cref{lem:base_kz_error}, in that an error with magnitude proportional to $(\norm{\b x} + \norm{\wb x}) u$ is incurred. However, in the accelerated setting, $\b w_k$ is scaled by the potentially large $\gamma$, which amplifies the error incurred by each iteration of the accelerated method.
Next, we compute this error by proving a recursive result on the expected norm of $\Delta_k \coloneqq  \| \widehat{\b v}_k - \b x\|_{\overline {\b P}_\lambda^\dagger}^2 + \f 1 {\tilde \mu} \| \widehat{\b y}_k - \b x \|^2.$

\begin{lemma}[ARK: Single iteration] \label{lem:accelvy}
    Let $\wb y_k$ and $\wb v_k$ be the current ARK iterates.
    Suppose the assumptions $ \m E \| \widehat{\b v}_k \|^2 \lesssim \f 1 {\tilde \mu} \| \b x \|^2$, $ \m E \| \widehat {\b y}_k \|^2 \lesssim \| \b x\|^2$ hold. In addition, suppose that $\tilde \mu \tilde \nu \le 1$, so that $\gamma \ge 1.$
    
    Then
    $$ \m E[ \Delta_{k+1}] \le \left(1 - \f 1 2 \sqrt{ \f {\tilde \mu}{\tilde \nu}} \right) \Delta_k + \left(1 - \f 1 2 \sqrt{\f {\tilde \mu}{\tilde \nu}}\right)^{1/2} \f { 2 C \norm{\b x } u}{\tilde \mu \sqrt{\tilde \nu}} \sqrt{\Delta_k} + \f 1 {\tilde \mu^2 \tilde \nu } C^2 \| \b x \|^2 u^2,$$
    for some $C(n)$ that is polynomial in $n$.
\end{lemma}
\begin{proof}

    First note that from basic stability results (\cref{thm:basic_stability}), we deduce
    $ \wb x_k = \tb x_k + \delta_1,$ where $\delta_1$ satisfies $\m E\norm{\delta_1}^2 \lesssim \norm{\b x}^2 u^2$.
    
    Next, observe
    $$ \widehat{ \b y}_{k+1} = \fl (\wb x_k - \wb w_k) = \tb x_k - \tb w_k + \delta_2,$$
    where $\delta_2 = (\wb x_k - \tb x_k) + (\wb w_k - \tb w_k) + \err(\wb x_k - \wb w_k)$ satisfies
    $\m E \norm{\delta_2}^2 \lesssim \norm{\b x}^2 u^2$ via an application of \cref{lem:wk}.
    
    Similarly,
    $$ \widehat{ \b v}_{k+1} = \fl(\beta \widehat{\b v}_k + (1 - \beta) \wb x_k - \gamma \widehat {\b w}_k) =  \beta \widehat{\b v}_k + (1 - \beta) \tb x_k - \gamma \tilde {\b w}_k + \delta_3,$$
    where
    $\delta_3 = (1 - \beta)(\wb x_k - \tb x_k) + (\gamma \tb w_k - \gamma \wb w_k) + \err(\beta \widehat{\b v}_k + (1 - \beta) \wb x_k - \gamma \widehat {\b w}_k)$ satisfies $\m E\norm{\delta_3}^2 \lesssim \gamma^2 \norm{\b x}^2 u^2$, again via an application of \cref{lem:wk}.

    Fix a $C(n)$ polynomial in $n$ such that $\m E \| \delta_2 \|^2 \le (C \| \b x \|u / \sqrt 2)^2$ and $ \m E \| \delta_3 \|^2 \le  (C \gamma\| \b x \|u / \sqrt 2)^2 .$ Now, since 
    $$ \wb y_{k+1} = \tb y_{k+1} + \delta_2, \enspace \wb v_{k+1} = \tb v_{k+1} + \delta_3,$$
    we have
    \begin{align*}
        \m E &[ \Delta_{k+1} ] 
        = \m E \left[ \| \tilde{\b v}_{k+1} - \b x + \delta_3 \|_{\overline {\b P}_\lambda^\dagger}^2 + \f 1 {\tilde \mu} \| \tilde{\b y}_{k+1} - \b x + \delta_2\|^2 \right], 
\end{align*}
which, after expanding and applying \cite[Lemma 2.3]{derezinski2025randomized}, satisfies the bound
\begin{multline}
        \m E [ \Delta_{k+1} ] \le \left(1 - \f 1 2 \sqrt{ \f {\tilde \mu}{\tilde \nu}} \right) \Delta_k + \m E \big[\| \delta_3 \|_{\overline {\b P}_\lambda^\dagger}^2\big] \\ +  \underbrace{\m E \big[2\|\delta_3 \|_{\overline {\b P}_\lambda^\dagger}  
        \| \tilde {\b v}_{k+1} - \b x \| _{\overline{\b P}_\lambda^\dagger}\big]}_{\text{(A)}} + \m E \big[ \tilde{\mu}^{-1} \| \delta_2 \|^2 \big] + \underbrace{\m E \big[2\tilde{\mu}^{-1} \| \delta_2 \| \| \tilde {\b y}_{k+1} - \b x \|\big]}_{\text{(B)}}. \label{eq:Delta_k+1}
\end{multline}
    For (A) and (B), we use Cauchy-Schwarz for random variables $X, Y$:
    $ \m E[XY] \le \sqrt{\m E[X^2]} \sqrt{ \m E[Y^2]}.$ 
    We deduce for (A) that 
    \begin{align*}
            \m E \left [2\|\delta_3 \|_{\overline {\b P}_\lambda^\dagger}  
        \| \tilde {\b v}_{k+1} - \b x \| _{\overline{\b P}_\lambda^\dagger}\right ]
         \le  \f {\sqrt 2 C \gamma \norm{\b x} u}{\sqrt{\tilde \mu}}  \sqrt{ \m E 
        \| \tilde {\b v}_{k+1} - \b x \|_{\overline{\b P}_\lambda^\dagger}^2},
    \end{align*}
    since $\norm{\overline{\b P}_\lambda^\dagger} \le 1 / \tilde \mu.$
    For (B), since we assume that $\gamma \ge 1,$
    \begin{align*}
        \m E\left [\f 2 {\tilde \mu} \| \delta_2 \| \| \tilde {\b y}_{k+1} - \b x \| \right ]
         \le \f {\sqrt 2 C \norm{\b x} u }{\tilde \mu} \sqrt{ \m E \| \tilde{\b y}_{k+1} - \b x\|^2} 
         \le \f {\sqrt 2 C \gamma \norm{\b x} u }{\tilde \mu} \sqrt{ \m E \| \tilde{\b y}_{k+1} - \b x\|^2}.
    \end{align*}
     We see that (A) + (B) satisfies
    \begin{align*}
        \m E  &\left [2\|\delta_3 \|_{\overline {\b P}_\lambda^\dagger}  
        \| \tilde {\b v}_{k+1} - \b x \| _{\overline{\b P}_\lambda^\dagger}\right ] + \m E\left [\f 2 {\tilde \mu} \| \delta_2 \| \| \tilde {\b y}_{k+1} - \b x \| \right ] 
        \\ & \le \f {\sqrt 2 C \gamma \norm{\b x} u}{\sqrt{\tilde \mu}} \left( \sqrt{ \m E 
        \| \tilde {\b v}_{k+1} - \b x \|_{\overline{\b P}_\lambda^\dagger}^2} + \f 1 {\sqrt{\tilde \mu}} \sqrt{ \m E \| \tilde{\b y}_{k+1} - \b x\|^2} \right) \\
        & {\le} \f {\sqrt 2 C \gamma \norm{\b x} u}{\sqrt{\tilde \mu}} \sqrt{2 \left ( { \m E 
        \| \tilde {\b v}_{k+1} - \b x \|_{\overline{\b P}_\lambda^\dagger}^2} + 
        \f 1 {\tilde \mu} { \m E \| \tilde{\b y}_{k+1} - \b x\|^2} \right)} \\ 
        &\le \left(1 - \f 1 2 \sqrt{ \f {\tilde \mu}{\tilde \nu}} \right)^{1/2} \f {2 C \gamma \norm {\b x} u}{\sqrt{\tilde \mu}} \sqrt{\Delta_k}.
    \end{align*}
        
    Thus, for the expectation in~\eqref{eq:Delta_k+1}, we have
    \begin{align*} 
    \m E &\left[\| \delta_3 \|_{\overline {\b P}_\lambda^\dagger}^2 + 2 \|\delta_3 \|_{\overline {\b P}_\lambda^\dagger}  
    \| \tilde {\b v}_{k+1} - \b x \| _{\overline{\b P}_\lambda^\dagger} + \f 1 {\tilde \mu} \| \delta_2 \|^2 + \f 2 {\tilde \mu} \| \delta_2 \| \| \tilde {\b y}_{k+1} - \b x \|\right] \\
    &\le \m E \left [ \| \delta_3 \|_{\overline {\b P}_\lambda^\dagger}^2 + \f 1 {\tilde \mu} \| \delta_2 \|^2 \right ] + \left(1 - \f 1 2 \sqrt{ \f {\tilde \mu}{\tilde \nu}} \right)^{1/2} \f {2 C \gamma \norm {\b x} u}{\sqrt{\tilde \mu}} \sqrt{\Delta_k}
    \\ 
    & \le \f {C^2 \gamma^2 \norm{\b x}^2 u^2} 2 \| \overline{\b P}_\lambda^\dagger\| + \f 1 {\tilde \mu} \f {C^2 \norm{\b x}^2 u^2} 2 + \left(1 - \f 1 2 \sqrt{ \f {\tilde \mu}{\tilde \nu}} \right)^{1/2} \f {2 C \gamma \norm {\b x} u}{\sqrt{\tilde \mu}} \sqrt{\Delta_k}
    \\ 
    &\le \f 1 {\tilde \mu} C^2 \gamma^2 \| \b x \|^2 u^2 + \left(1 - \f 1 2 \sqrt{ \f {\tilde \mu}{\tilde \nu}} \right)^{1/2} \f {2 C \gamma \norm {\b x} u}{\sqrt{\tilde \mu}} \sqrt{\Delta_k}, \numberthis \label{eq:Delta_k+1_rightmost}
    \end{align*}
    since $\| \overline{\b P}_\lambda^\dagger\| \le 1 / \tilde \mu$.
    Therefore, by combining~\eqref{eq:Delta_k+1} and \eqref{eq:Delta_k+1_rightmost}, we conclude
    $$ \m E[ \Delta_{k+1}] \le \left(1 - \f 1 2 \sqrt{ \f {\tilde \mu}{\tilde \nu}} \right) \Delta_k + \left(1 - \f 1 2 \sqrt{\f {\tilde \mu}{\tilde \nu}}\right)^{1/2} \f { 2 C \gamma \norm{\b x } u}{\sqrt{\tilde \mu}} \sqrt{\Delta_k} + \f 1 {\tilde \mu} C^2\gamma^2 \| \b x \|^2 u^2.$$
    Substituting in $\gamma = 1 / \sqrt{\tilde \mu \tilde \nu}$ gives the desired result.
\end{proof}

The large error in $\gamma \b w_k$ is reflected in the result in \cref{lem:accelvy}, as we see a convergence horizon term of $C^2 \norm{\b x}^2 u^2 / (\tilde \mu ^2 {\tilde \nu}).$ Approximating $\tilde \mu \approx \tilde \kappa(\mat A)^{-2}$, this implies that the forward error will necessarily have at least a $\tilde \kappa(\mat A)^{2}$ term in its convergence horizon, potentially amplified by slow convergence. Fortunately, the accelerated method converges quickly enough so that it is not amplified much, as implied by \cref{thm:accelkz}.

\begin{proof}[Proof of \cref{thm:accelkz}]
    For all $i \le k,$ let $e_i = \m E[\Delta_i] = \m E[ \| \widehat{\b v}_i - \b x \|_{\overline{\b P}_\lambda^\dagger}^2 + \f 1 {\tilde \mu} \| \widehat{\b y}_i - \b x \|^2]$.  We inductively assume that the conditions of \cref{lem:accelvy} hold, ie. for all $i < k$, we may assume that $ \m E \| \widehat{\b v}_i \|^2 \lesssim \f 1 {\tilde \mu} \| \b x \|^2$ and  $ \m E \| \widehat {\b y}_i \|^2 \lesssim \| \b x\|^2$.
    Then, from the lemma, we get the recursive inequality, for all $i \le k$,
    $$ e_{i} \le r e_{i-1} + 2 \sqrt{d r } \sqrt{e_{i-1}} + d,$$
    where we set $r: = 1 - \f 1 2 \sqrt{ \tfrac {\tilde \mu}{\tilde \nu}}$ and $d \coloneqq  C^2 \| \b x \|^2 u^2/ (\tilde \mu^2 \tilde \nu).$
   Consider the map $f: \m R_{\ge 0} \to \m R_{\ge 0}$ defined by $f(e) = r e + 2\sqrt {dre} + d$, which has fixed point $e_{\text{fix}} \coloneqq  \f d {(1 - \sqrt r)^2}.$
    Note that $f$ is increasing, so $f$ maps $[e_{\text{fix}}, \infty)$ to itself. Also, on this interval, the derivative of $f$ satisfies 
    $$ f'(e) \le r + \f {\sqrt{dr}}{\sqrt{e_{\text{fix}}}} = r + \sqrt r (1 - \sqrt r) = \sqrt r < 1,
    $$ 
    Then, the mean value theorem implies that $f$ is Lipschitz with constant $\sqrt r$. Let $e_i' = f^i(e_0)$. We deduce
    \begin{align*}
    &| f(e_{k-1}') - f(e_{\text{fix}}) | \le \sqrt r | e_{k-1}' - e_{\text{fix}}| 
    \\ \implies &| e_k' - e_{\text{fix}}| \le r^{k/2} | e_0 - e_{\text{fix}} |  
    \\ \implies & e_k' \le e_{\text{fix}} + r^{k/2} |e_0 - e_{\text{fix}}| \le e_{\text{fix}} + r^{k/2} e_0,
    \end{align*} 
    as long as $e_0 \in [e_{\text{fix}}, \infty).$ On the other hand, if $e_0 < e_{\text{fix}}$, then $e_k' < e_{\text{fix}}$ for all $k$ because $f$ is increasing. Thus, in either case, we conclude that 
    $ e_k \le e_k' \le e_{\text{fix}} + r^{k/2} e_0,$
    which implies
    \begin{equation*}
        \begin{cases}
            \m E\| \widehat{\b v}_k - \b x \|^2  \le e_{\text{fix}} + r^{k/2} e_0, \\ 
            \m E\| \widehat{\b y}_k - \b x \|^2  \le \tilde \mu (e_{\text{fix}} + r^{k/2} e_0).
        \end{cases}
    \end{equation*}
    Notice that the conditions $ \m E \| \widehat{\b v}_k \|^2 \lesssim \f 1 {\tilde \mu} \| \b x \|^2$ and  $ \m E \| \widehat {\b y}_k \|^2 \lesssim \| \b x\|^2$ necessary for the induction easily follow from these inequalities, since the two inequalities imply that $\m E \norm{ \wb v_k - \b x  }^2 \lesssim \f 1 {\tilde \mu} \norm{\b x}^2$ and $\m E \norm{\wb y_k - \b x}^2 \lesssim \norm{\b x}^2$, after which the triangle inequality implies the desired conditions.
    Finally, observe that for some constant $c$ (independent of $n$) and some polynomial $C_1(n),$
    \begin{align*}
        \m E \|\widehat{\b x}_k - \b x\|^2  
        &\le  c^2 \cdot \m E \| \alpha(\widehat{\b v}_k - \b x ) \|^2 + c^2 \cdot \m E \| (1 - \alpha^2)(\widehat{\b y}_k - \b x) \|^2 + C_1 \| \b x \|^2 u^2 \\ 
        & \le c^2 (\alpha^2 + \tilde \mu) (e_{\text{fix}} + r^{k/2} e_0) +  C_1 \| \b x \|^2 u ^2 \\ 
        & \le \f {C^2 \| \b x \|^2 u^2}{(1 - \sqrt r)^2\tilde \mu \tilde \nu} + c^2 r^{k/2} \tilde \mu e_0,
    \end{align*}
    where we redefined $C$ and used $\alpha^2 \le \tilde \mu$. With the initialization $\b v_0 = \b y_0 = \b 0$, we have that 
    $ e_0 = \m E[ \| \widehat{\b v}_0 - \b x \|_{\overline{\b P}_\lambda^\dagger}^2 + \f 1 {\tilde \mu} \| \widehat{\b y}_0 - \b x \|^2] \le \f 2 {\tilde \mu} \| \b x \|^2.$ Hence,
    \begin{align*} \m E \|\widehat{\b x}_k - \b x\|^2 \le \f {C^2\| \b x \|^2u^2}{(1 - \sqrt r)^2 \tilde \mu \tilde \nu} + c^2r^{k/2} \| \b x \|^2.
     \end{align*}
      Lastly, $\sqrt r = \sqrt{1 - \f 1 2 \sqrt{\tfrac {\tilde \mu} {\tilde \nu}}} \le 1 - \f 1 4 \sqrt{\tfrac {\tilde \mu} { \tilde \nu}}$ simplifies the first term, and $r \le 1 - \f {1}{4 \tilde \kappa(\mat A)\sqrt {n} }$ (from option \ref{op:1} or \ref{op:2}) simplifies the second term.
\end{proof}

Lastly, we combine the black-box result \cref{thm:refinement_exp} with \cref{thm:accelkz} to prove \cref{thm:accelkz_ir}.
\begin{proof}[Proof of \cref{thm:accelkz_ir}]
    Let $k_i$ be the number of iterations of ARK in the $i$th refinement, and let $c$ be the constant and let $C_1(n)$ be the polynomial so that~\cref{thm:accelkz} gives
    \begin{align*}
     \m E \|\widehat{\b x}_1 - \b x\| \le 
     \left[ \left(1 - \f 1 {4 \tilde \kappa(\mat A) \sqrt{n}} \right)^{k_1/4}c +  C_1 \tilde \kappa(\mat A)^2 u\right] \norm{\b x}.
     \end{align*}

    Then, during the $i$th refinement, $ \left[ \left(1 - \f 1 {4 \tilde \kappa(\mat A) \sqrt{n}} \right)^{k_i/4}c + C_1 \tilde{\kappa}(\mat{A})^2 u \right]$ factor is the $q$ defined in~\cref{thm:refinement_exp}. The theorem will immediately give the desired result, so it suffices to show that $(1 + C_2u) q < 1/2$ and $C_2 \tilde \kappa(\mat A) u < 3/4$, where we set $C_2$ to be the polynomial $C$ defined in~\cref{thm:refinement_exp}. If we choose $k_i$ to satisfy
    $$k_i > \f {4 \log \left(\f  1 {2c(1 + C_2u)} - \f{C_1}{c} \tilde \kappa(\mat{A})^2 u \right)} { \log \left(1 - \f 1 {4 \tilde \kappa(\mat A) \sqrt{n}} \right)} = \order(\tilde \kappa(\mat A) \sqrt n),$$ then
    \begin{align*} 
    (1 + C_2u) q 
    < (1 + C_2u) \left[ \left(\f  1 {2(1 + C_2u)} - C_1 \tilde \kappa(\mat{A})^2 u \right) + C_1\tilde \kappa(\mat{A})^2 u \right]
    = \f 1 2.
    \end{align*}
    The second requirement, $C_2 \tilde \kappa(\mat A) u < 3/4$, is implied by the assumption $\tilde \kappa(\mat{A})^2 u \ll 1.$
\end{proof}

\section{Experimental results}\label{sec:experiment}
In this section, we generate matrices with a variety of singular value distributions to experimentally verify the theoretical results in \cref{thm:base_kz}, \cref{thm:base_kz_ir}, \cref{thm:accelkz}, and \cref{thm:accelkz_ir}. Code used to generate the results can be viewed at \cite{kaczmarz_stability_repo}. 

\begin{table}[t]
\centering
\caption{Test matrices with prescribed singular value distributions $\sigma_1 \ge \sigma_2 \ge ... \ge \sigma_n$. Parameters $\alpha$ and $\beta$ are chosen to match the desired Demmel condition number.}
\begin{tabular}{ll}
\toprule
\textbf{Type} & \textbf{Singular Value Formula} \\
\midrule
exp 
& $\sigma_k = \exp(-\beta (k-1))$ \\

poly 
& $\sigma_k = k^{-\alpha}$ \\


highrank 
& $\sigma_k =
\begin{cases}
\text{linear decay from } 1 \text{ to } \tilde{\kappa}^{-1}, & 1 \le k \le 0.9n \\
\tilde{\kappa}^{-1}, & 0.9n < k \le n
\end{cases}$ \\

harmonic 
& $\sigma_k = (1 + \alpha (k-1))^{-1}$ \\
\bottomrule
\end{tabular}
\label{table:test_suite}
\end{table}

\vspace{0.5em}
\noindent  \textbf{Setup.} We run experiments on a test suite consisting of four $n \times n$ randomly generated matrices with a specified Demmel condition number $\tilde \kappa(\mat{A})$.
We generate the test suite using the singular value distributions in \cref{table:test_suite} with Haar-random left and right singular vectors.
We use $n = 500$ and $\tilde \kappa(\mat{A}) \in \{10^4, 10^5\}$.
The solution vector $\b x$ is chosen to be a uniformly random unit vector, and we set $\b b \coloneqq \mat{A} \b x$.
All experiments were run in single precision, which has $u \approx 10^{-8}.$ 

We test randomized Kaczmarz and ARK both with and without iterative refinement.
For randomized Kaczmarz, we perform one step of iterative refinement initiated halfway through the run.
For ARK, we initiate a new step of iterative refinement every $10^8$ iterations.
As a baseline, we compare against the built-in MATLAB command $\mat A \backslash \b b$, which solves $\mat{A}\vec{x} = \vec{b}$ using a standard direct solver.

When $\tilde \kappa(\mat{A}) = 10^4$, we run each method for $10^9$ iterations because it is approximately the number of iterations needed for randomized Kaczmarz to converge.
Note that $10^9$ is a large number of iterations---already bordering on infeasibility---which highlights the impracticality of plain randomized Kaczmarz when the system is ill-conditioned.
For $\tilde \kappa = 10^5$, we run each method for $10^{10}$ iterations to clearly see the convergence of ARK with iterative refinement.

For ARK, we regularize the probabilities as described in the beginning of \cref{sec:rak}.
For the random matrices we consider, the rows are typically balanced in norm already, so no regularization should be necessary.
We use a small value of $\lambda = 10^{-6}$ anyway to guard against rare events in the random test matrix generation.
In our experience, we did not notice much difference when we skipped regularization by setting $\lambda = 0$. 
For the acceleration parameters $\tilde \mu$ and $\tilde \nu$, we set them to the exact values of $\mu$ and $\nu$ (see \eqref{eq:mu_nu}). We choose these parameters to understand the stability of ARK in the most favorable scenario; see \cite{derezinski2025randomized} for how to estimate these parameters at runtime.


\vspace{0.5em}
\noindent \textbf{Results.}
\cref{fig:exp} in the introduction shows the result of these methods on the exp matrix for both $\tilde \kappa(\mat{A}) \in \{10^4, 10^5\}$ cases. 
\Cref{fig:testsuite1e4,fig:testsuite1e5} depict the rest of the results of these methods on the generated matrices. In the legends for the figures, RK stands for randomized Kaczmarz, ARK is accelerated Kaczmarz as usual, and ``$+ $ IR" signifies combining the specified method with iterative refinement.

Consistent with \cref{thm:base_kz}, in all figures, randomized Kaczmarz consistently achieves errors that are orders of magnitude worse than the direct method, confirming the method is not forward stable.
For the $\tilde \kappa(\mat{A}) = 10^4$ test suite, we expect that performing iterative refinement once every $\order(\tilde \kappa(\mat A)^2)$ randomized Kaczmarz iterations should be sufficient to obtain forward stability.
Indeed, we perform iterative refinement a single time after $5\cdot \tilde \kappa(\mat A)^2 = 5\cdot 10^8$ steps, and we obtain forward stability on the entire $\tilde \kappa(\mat{A}) = 10^4$ test suite.
When $\tilde \kappa(\mat{A}) = 10^5$, randomized Kaczmarz does not converge in the allotted $10^{10}$ iteration budget, and the resulting error is large with and without iterative refinement.



In all figures, ARK converges significantly more rapidly than standard randomized Kaczmarz, converging to its maximum numerically achievable accuracy within the specified iteration count on all problems except the ``harmonic'' test problem with $\tilde\kappa(\mat{A}) = 10^5$.
However, without iterative refinement, ARK is still observed to be not fully stable---indeed, on two of four instances with $\tilde\kappa(\mat{A}) = 10^4$, the accuracy of ARK stagnates above that of standard randomized Kaczmarz!
With iterative refinement every $10^8$ iterations, we obtain convergence to be comparable with the direct method in all cases when ARK converges sufficiently rapidly. 

\begin{figure}[t]
\centering
\begin{minipage}{0.325\linewidth}
    \centering
    poly
  \includegraphics[width=\linewidth]{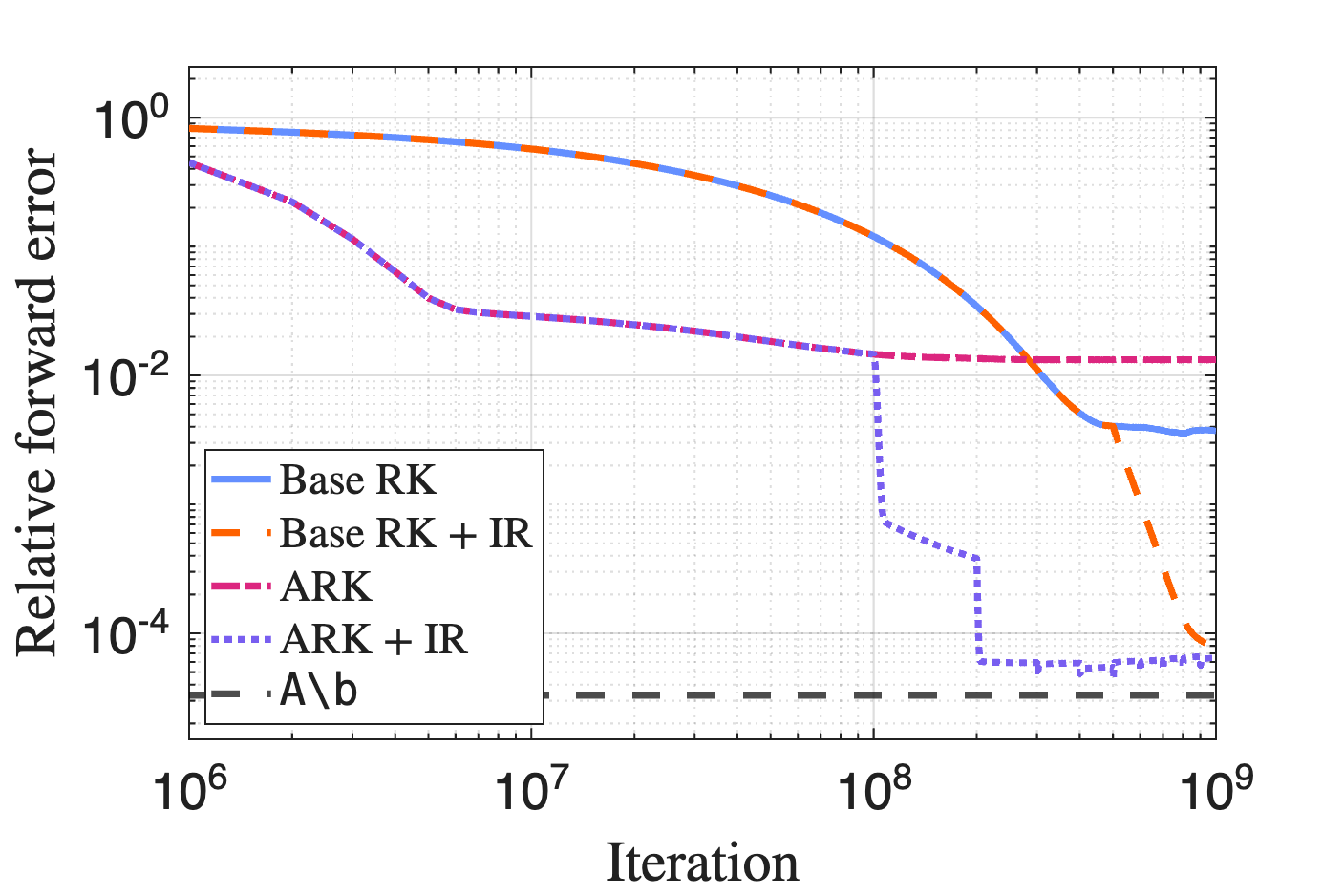}
\end{minipage}
\begin{minipage}{0.325\linewidth}
    \centering
    highrank
\includegraphics[width = \linewidth]{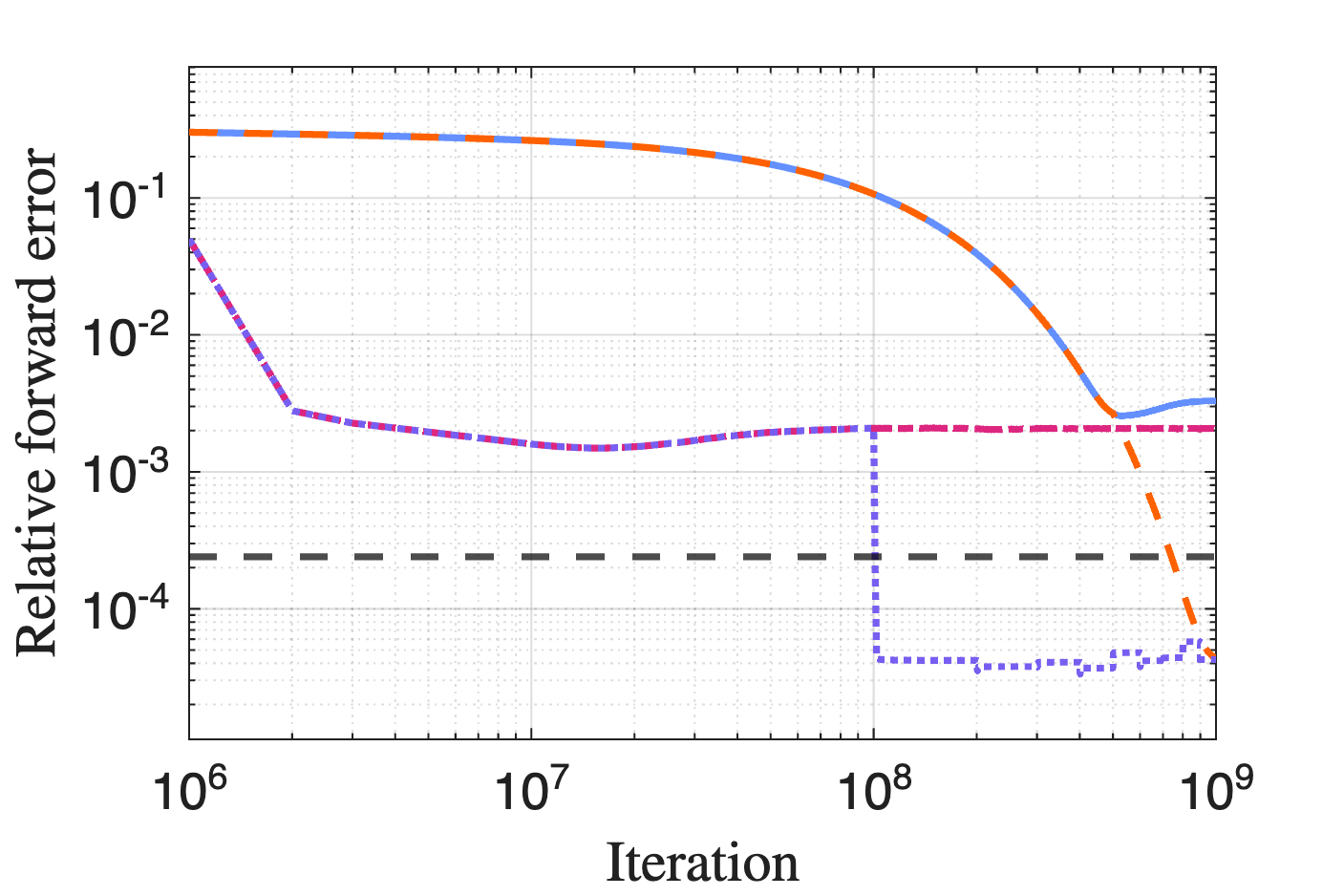} 
\end{minipage}
\begin{minipage}{0.325\linewidth}
    \centering
    harmonic
\includegraphics[width = \linewidth]{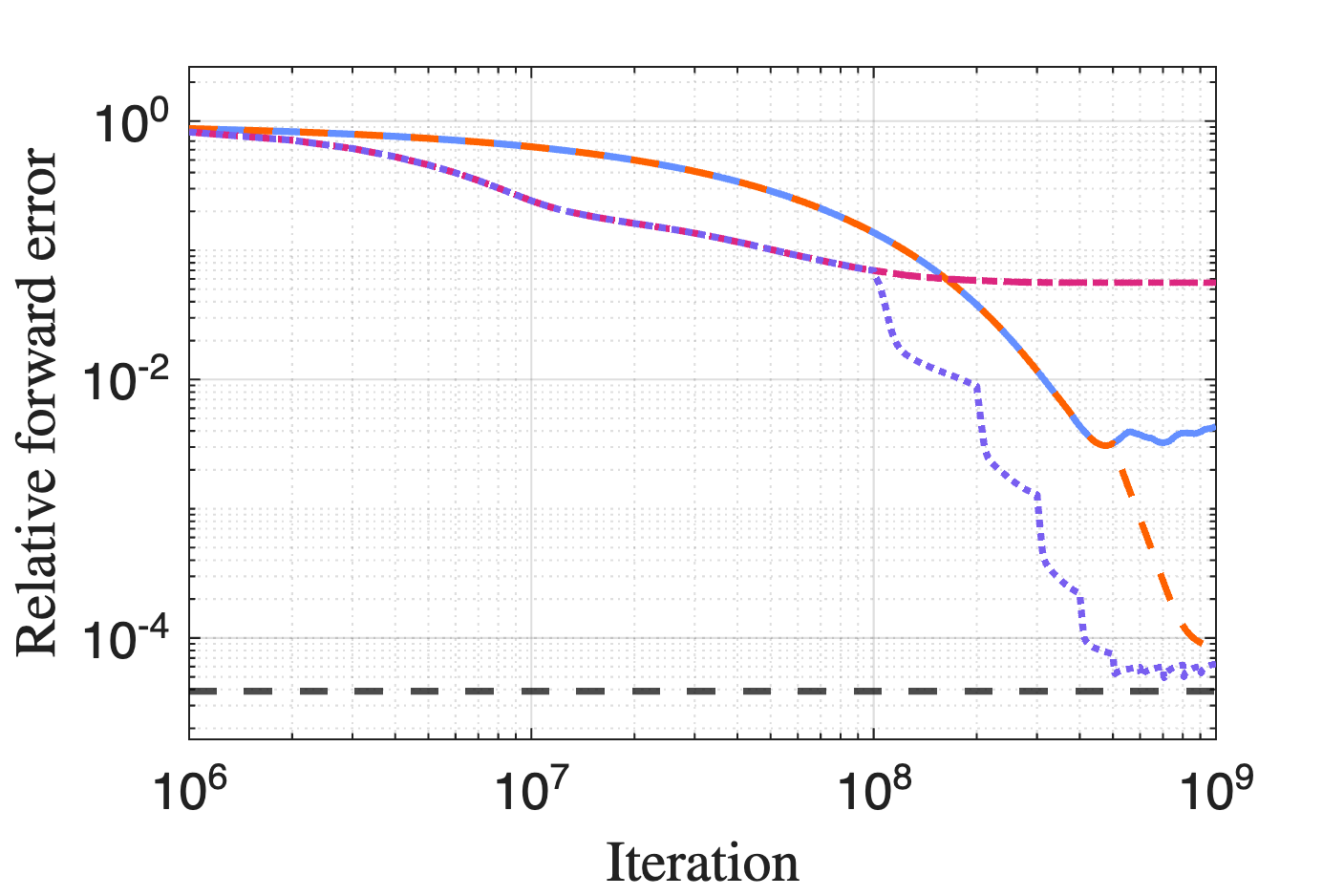} 
\end{minipage}
\caption{Forwards error for test suite with $\tilde{\kappa}(\mat{A}) = 10^{4}.$}
\label{fig:testsuite1e4}
\end{figure}

\begin{figure}[t]
\centering

\begin{minipage}{0.33\linewidth}
  \centering 
  poly
  \includegraphics[width=\linewidth]{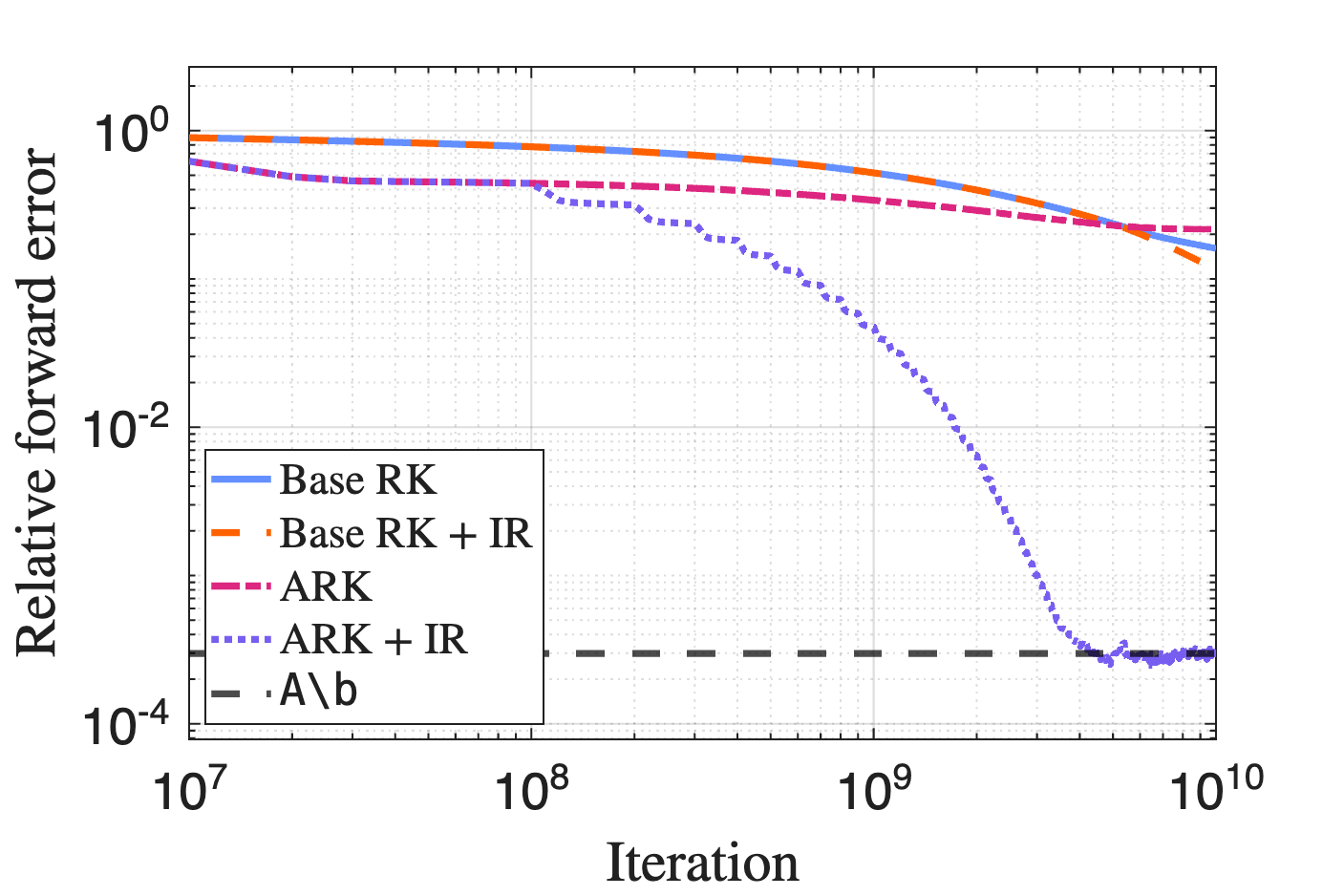}
  \end{minipage}
\begin{minipage}{0.325\linewidth}
\centering
highrank
\includegraphics[width = \linewidth]{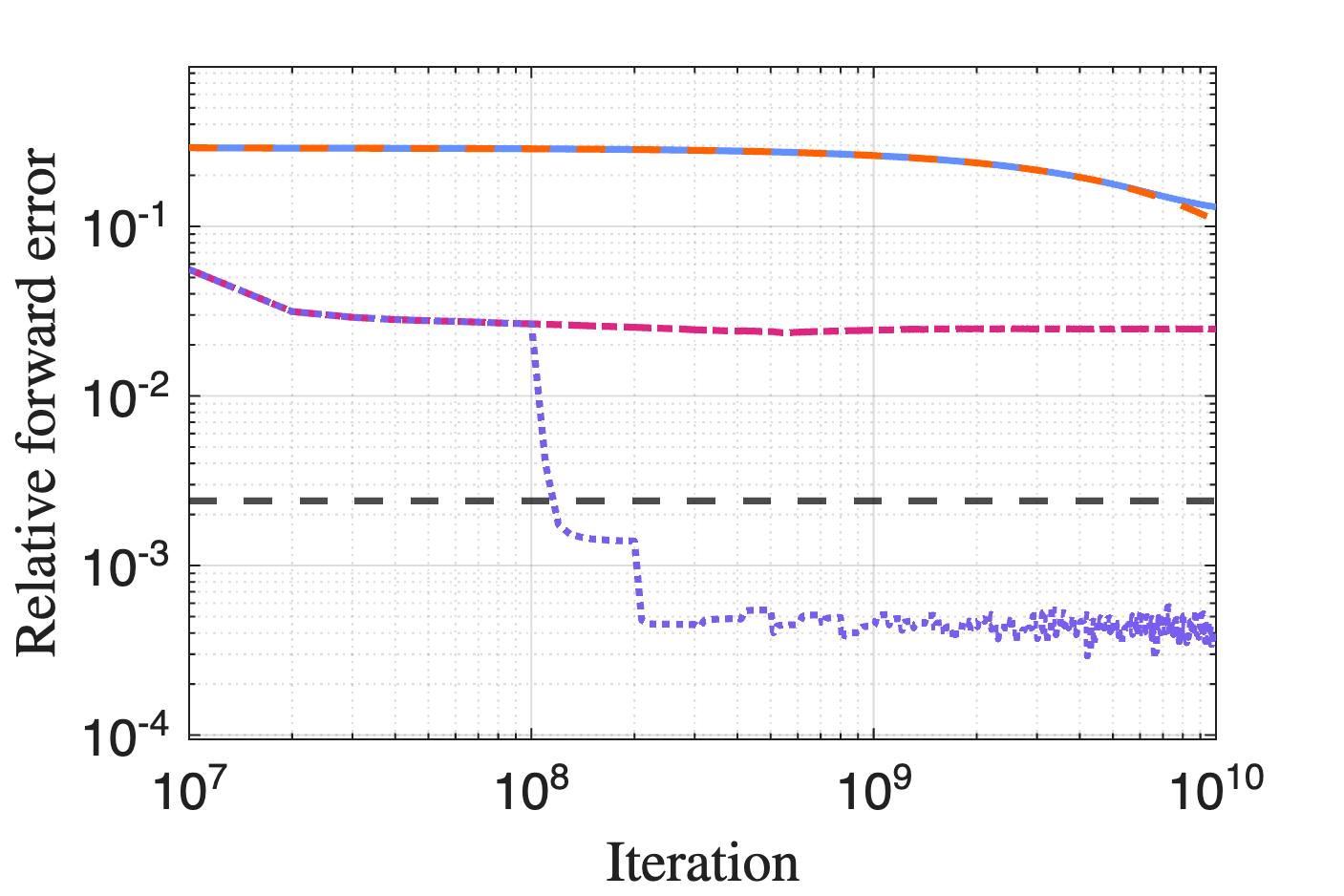}
\end{minipage}
\begin{minipage}{0.325\linewidth}
\centering
harmonic
\includegraphics[width = \linewidth]{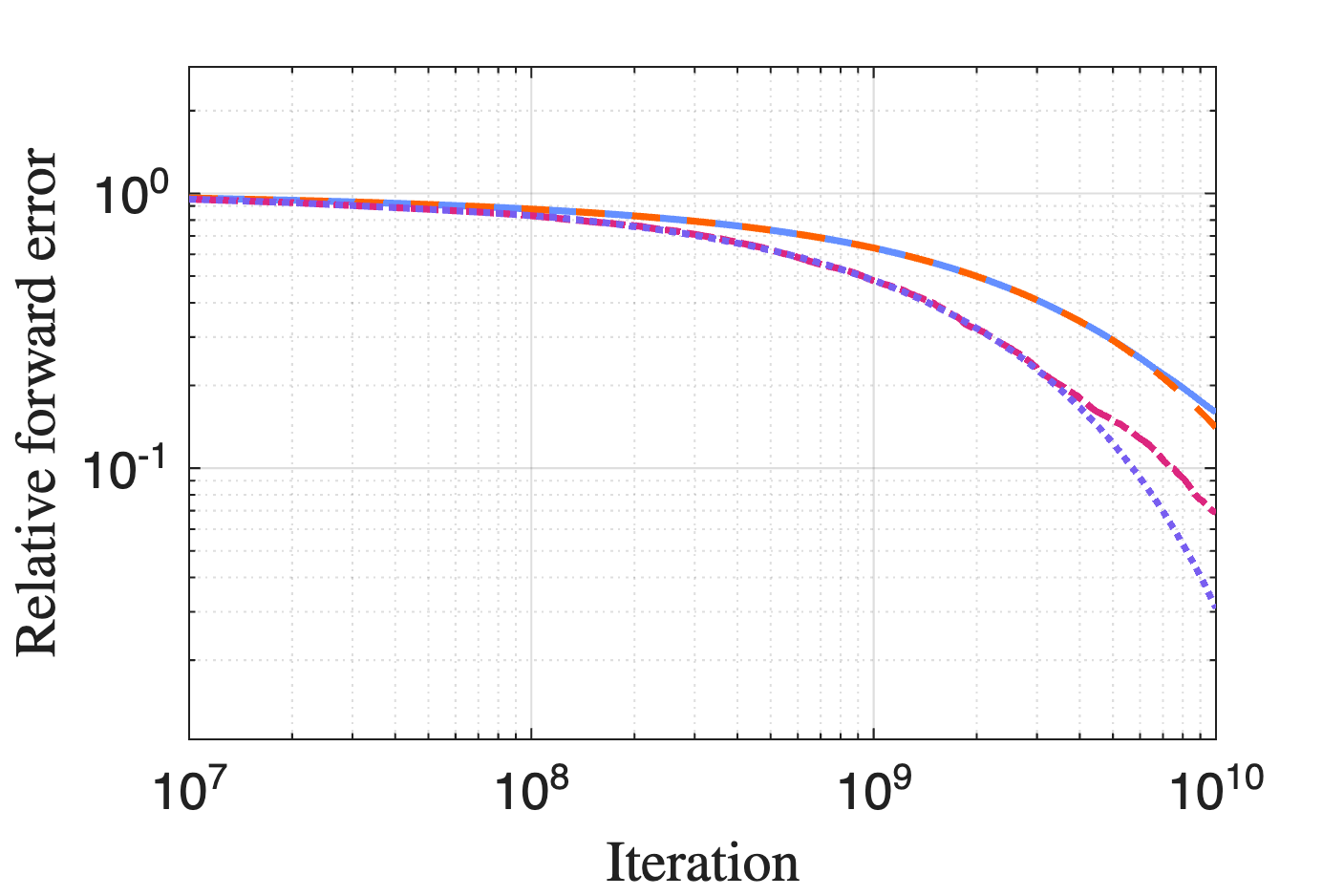}
\end{minipage}
\caption{Forwards error for test suite with $\tilde{\kappa}(\mat{A}) = 10^{5}.$}
\label{fig:testsuite1e5}
\end{figure}
\FloatBarrier

\section{Conclusion} \label{sec:conclusion}
In this work, we investigated the numerical stability of randomized Kaczmarz methods and showed how those stability properties can be improved by iterative refinement.
Our results provide both a conceptual and practical framework for understanding how such methods behave in finite-precision arithmetic and how their limitations can be addressed.


Our first results treat the randomized Kaczmarz method.
We provide numerical experiments demonstrating that the method is not forward stable, and we complement it with theoretical analysis showing that the relative forward error can stagnate as high as $\tilde{\kappa}(\mat{A})^2u$.
We show that this issue can be remedied by combining randomized Kaczmarz with iterative refinement.
Our analysis of iterative refinement is general and treats an arbitrary randomized linear solver.


Given the slow convergence of randomized Kaczmarz on ill-conditioned problems, we next turned our attention accelerated randomized Kaczmarz (ARK).
Our numerical and theoretical results suggest ARK is also not forward stable, showing that the relative forward error stagnates at the same $\tilde{\kappa}(\mat{A})^2u$ level. 
As with base randomized Kaczmarz, we can stabilize with iterative refinement.

There are several natural directions to investigate in future work.
First, one could investigate the numerical stability of block randomized Kaczmarz methods \cite{Elfving1980,Needell_2014,derezinski2024solving,derezinski2025randomized}. 
Each step of block randomized Kaczmarz requires accessing the pseudoinverse of a wide rectangular matrix, which introduces additional complications in the analysis and possibilities for new stability phenomena to emerge.
Second, one could extend this analysis to finite-precision analysis of other randomized iterative methods.
For example, do these instabilities manifest in stochastic gradient descent over nonlinear functions, and if so, can these methods be stabilized by a form of iterative refinement?



\appendix
\section{Appendix}
In this appendix, we prove that (accelerated) randomized Kaczmarz with iterative refinement is equivalent to (accelerated) randomized Kaczmarz without iterative refinement (\cref{sec:ir_exact}), and we prove \cref{lem:mu_nu} (\cref{sec:mu_nu}).

\subsection{Iterative refinement in exact arithmetic} \label{sec:ir_exact}
In this section, we show that the iterates produced by Kaczmarz with iterative refinement and without are the same in exact arithmetic, and we prove the same result for ARK.

We first explain what it means for the iterates to be the same.
Denote the iterates produced by Kaczmarz without iterative refinement by $\b x_0, \b x_1, \dots.$ Let $t$ be the number of refinements performed, and let $k_i$ be the number of Kaczmarz iterations performed before doing each refinement, for all $0 \le i \le t - 1.$ Let $s_i = \sum_{j = 0}^{i-1} k_j$.
Denote the iterates produced by the refinement scheme by 
\begin{equation}
\begin{split}
     \b x_{s_0 + 0}^{(0)}, \dots ,&\b x_{s_0 + k_0}^{(0)} = \b x_{s_1}^{(1)},\\
     \b x_{s_1 + 1}^{(1)}, \dots, &\b x_{s_1 + k_1}^{(1)} = \b x_{s_2}^{(2)}, \\ 
     &\vdots \\ 
    \b x_{s_t + 1}^{(t)}, &\dots, \b x_{s_{t+1}}^{(t)}. \label{eq:ir_notation}
\end{split}
\end{equation}
    Here, we use the superscripts to track the refinement number, and we say that $\b x_{s_i + k_i}^{(i)} = \b x_{s_{i+1}}^{(i+1)}$ since the last iteration during a refinement is equal to the initial iteration of the subsequent refinement.
    
    Then, to say that the iterates produced by Kaczmarz with and without iterative refinement are the same is to say that 
    \begin{equation}
    \label{eq:same_kz}
    \b x_{s_i + j} = \b x_{s_i + j}^{(i)},
    \end{equation}
    for all $0 \le i \le t$ and $0 \le j \le k_i.$

\begin{lemma} \label{lem:ir_kz_exact}
    The iterates produced by Kaczmarz with iterative refinement and without are the same in exact arithmetic, i.e. \eqref{eq:same_kz} is satisfied for all $0 \le i \le t$ and $0 \le j \le k_i.$
\end{lemma}

\begin{proof}
    We induct on the refinement number $i$. Assume that the two schemes produce the same iterates during the $i$th refinement, i.e. we have 
    $$ \b x_{s_i + j} = \b x_{s_i + j}^{(i)},$$
    for all $0 \le j \le k_i.$ Then, during the $(i+1)$st refinement, we initially have
    $ \b x_{s_{i+1}} = \b x_{s_{i+1}}^{(i+1)}.$
    Assume via induction that for some $j \ge 0$,
    $$ \b x_{s_{i+1} + j } = \b x_{s_{i+1} + j}^{(i+1)}.$$
    The $(i+1)$st refinement iterates on the system $\mat A \b x = \b b - \mat A \b x_{s_{i+1}}.$
    The underlying iterates produced by Kaczmarz during this $(i+1)$st refinement are added to $\b x_{s_{i+1}}$ in order to produce the iterates generated by the refinement scheme.
    
   Therefore, with $b = b_{r(i+1)}$ and $\b a = \b a_{r(i+1)}$, the next Kaczmarz iteration is
    \begin{align*}
        &\left[\b x_{s_{i+1}+(j+1)}^{(i+1)} - \b x_{s_{i+1}} \right] \\
        &\quad= \left[\b x_{s_{i+1} + j}^{(i+1)} - \b x_{s_{i+1}}\right] + \f { \left[b - \langle \b a, \b x_{s_{i+1}} \rangle\right] - \left \langle \b a, \left[\b x_{s_{i+1} + j}^{(i+1)} - \b x_{s_{i+1}}\right] \right \rangle } { \langle \b a, \b a \rangle } \b a \\
        &\quad= \left[\b x_{s_{i+1} + j} - \b x_{s_{i+1}}\right]  + \f { b - \langle \b a, \b x_{s_{i+1} + j}^{(i+1)} \rangle} { \langle \b a, \b a \rangle } \b a = \left[\b x_{s_{i+1} + j} - \b x_{s_{i+1}}\right]  + \f { b - \langle \b a, \b x_{s_{i+1} + j} \rangle} { \langle \b a, \b a \rangle } \b a \\
       &\quad= \b x_{s_{i+1} + j} + \f { b - \langle \b a, \b x_{s_{i+1} + j} \rangle} { \langle \b a, \b a \rangle } \b a  - \b x_{s_{i+1}} = \b x_{s_{i+1}+ (i+1)} - \b x_{s_{i+1}}.
    \end{align*}
    Hence 
    $$ \b x_{s_{i+1}+(j+1)}^{(i+1)} = \b x_{s_{i+1} + (j+1)},$$
    as desired.
    
\end{proof}

For ARK, we use the same notation for the $\b x$ iterates as in \eqref{eq:ir_notation}, and we adopt the same notation for the $\b y, \b v$, and $\b w$ iterates as well.
Then, to say that the iterates produced by ARK with and without iterative refinement are the same is to say that 
    \begin{equation}
    \label{eq:same_akz}
    \b x_{s_i + j} = \b x_{s_i + j}^{(i)}, \enspace \b y_{s_i + j} = \b y_{s_i + j}^{(i)}, \enspace \b v_{s_i + j} = \b v_{s_i + j}^{(i)}, \enspace \b w_{s_i + j} = \b w_{s_i + j}^{(i)},
    \end{equation}
    for all $0 \le i \le t$ and $0 \le j \le k_i.$ 

Throughout this paper, we have been analyzing ARK with block size one. However, for the following lemma, we present the result for a general block size. If $S$ is the subset of row indices we iterate on, then the projection vector now takes the form
$$ \b w = \mat{A}_S^\top  ( \mat{A}_S  \mat{A}_S^\top  + \lambda \mat{I})^{-1}( \mat{A}_S \b x - \b b_S).$$

\begin{lemma} \label{lem:ir_accelkz_exact}
    The iterates produced by ARK (with any block size) with iterative refinement and without are the same in exact arithmetic, i.e. \eqref{eq:same_akz} is satisfied for all $0 \le i \le t$ and $0 \le j \le k_i.$
\end{lemma}
    
\begin{proof}
    The proof is similar to the proof of \cref{lem:ir_kz_exact}.
    We induct on the refinement number $i$. Assume that the two schemes produce the same iterates during the $i$th refinement, i.e. we have 
    $$ \b y_{s_i + j} = \b y_{s_i + j}^{(i)}, \enspace \b v_{s_i + j} = \b v_{s_i + j}^{(i)}, \enspace \b x_{s_i + j} = \b x_{s_i + j}^{(i)}, \enspace \b w_{s_i + j} = \b w_{s_i + j}^{(i)},$$
    for all $0 \le j \le k_i.$ Then, during the $(i+1)$st refinement, we initially have
    $$ \b y_{s_{i+1}} = \b y_{s_{i+1}}^{(i+1)}, \enspace  \b v_{s_{i+1}} = \b v_{s_{i+1}}^{(i+1)}, \enspace \b x_{s_{i+1}} = \b x_{s_{i+1}}^{(i+1)}, \enspace \b w_{s_{i+1}} = \b w_{s_{i+1}}^{(i+1)}.$$
    Assume via induction that for some $j \ge 0$,
    $$ \b y_{s_{i+1} + j } = \b y_{s_{i+1} + j}^{(i+1)}, \enspace \b v_{s_{i+1} + j } = \b v_{s_{i+1} + j}^{(i+1)}, \enspace \b x_{s_{i+1} + j } = \b x_{s_{i+1} + j}^{(i+1)}, \enspace \b w_{s_{i+1} + j } = \b w_{s_{i+1} + j}^{(i+1)}.$$
    The $(i+1)$st refinement iterates on the system $\mat A \b x = \b b - \mat A \b x_{s_{i+1}}.$
    The $\b y, \b v,$ and $\b x$ underlying iterates produced by ARK during this $(i+1)$st refinement are added to $\b y_{s_{i+1}}, \b v_{s_{i+1}},$ and $\b x_{s_{i+1}}$ respectively in order to produce the iterates generated by the refinement scheme.
    
   Therefore, with $b = b_{r(i+1)}$ and $a = \b a_{r(i+1)}$, the next ARK iterate for $\b y$ is
    \begin{align*}
        \left[ \b y_{s_{i+1} + (j+1)}^{(i+1)} - \b y_{s_{i+1}} \right] &= 
        \left[ \b x_{s_{i+1} + j}^{(i+1)} - \b y_{s_{i+1}} \right] - \b w_{s_{i+1} + j}^{(i+1)} \\ 
        &= \left[ \b x_{s_{i+1} + j} - \b y_{s_{i+1}} \right] - \b w_{s_{i+1} + j} = \b y_{s_{i+1} + (j+1)} - \b y_{s_{i+1}}.
    \end{align*}
    Hence $\b y_{s_{i+1} + (j+1)}^{(i+1)} = \b y_{s_{i+1}+(j+1)}.$
   
    The next ARK iterate for $\b v$ is
    \begin{align*}
    &\left[\b v_{s_{i+1} + (j+1)}^{(i+1)}  - \b v_{s_{i+1}} \right]  = \beta \left[ \b v_{s_{i+1}+ j}^{(i+1)}  - \b v_{s_{i+1}} \right] + (1 - \beta) \left[\b x_{s_{i+1} + j}^{(i+1)}  - \b v_{s_{i+1}}\right] - \gamma \b w_{s_{i+1} + j}^{(i+1)}
    \\&\qquad= \beta \b v_{s_{i+1}+j}^{(i+1)} + (1 - \beta) \b x_{s_{i+1}+j}^{(i+1)}  - \b v_{s_{i+1}} - \gamma \b w_{s_{i+1}+j}^{(i+1)}
    \\&\qquad= \beta \b v_{s_{i+1}+j} + (1 - \beta) \b x_{s_{i+1}+j} - \gamma \b w_{s_{i+1}+j} - \b v_{s_{i+1}} = \b v_{s_{i+1}+(j+1)} - \b v_{s_{i+1}}.
    \end{align*}
    Hence $\b v_{s_{i+1} + (j+1)}^{(i+1)} = \b v_{s_{i+1}+(j+1)}.$

    The next ARK iterate for $\b x$ is
    \begin{align*}
    &\left[\b x_{s_{i+1} + (j+1)}^{(i+1)}  - \b x_{s_{i+1}} \right]  = \alpha \left[ \b v_{s_{i+1}+ j}^{(i+1)}  - \b x_{s_{i+1}} \right] + (1 - \alpha) \left[\b y_{s_{i+1} + j}^{(i+1)}  - \b x_{s_{i+1}}\right]
    \\&\qquad= \alpha \b v_{s_{i+1}+j}^{(i+1)} + (1 - \alpha) \b y_{s_{i+1}+j}^{(i+1)}  - \b x_{s_{i+1}}
    \\&\qquad= \alpha \b v_{s_{i+1}+j} + (1 - \alpha) \b y_{s_{i+1}+j}  - \b x_{s_{i+1}} = \b x_{s_{i+1}+(j+1)} - \b x_{s_{i+1}}.
    \end{align*}
    Hence $\b x_{s_{i+1} + (j+1)}^{(i+1)} = \b x_{s_{i+1}+(j+1)}.$

    Lastly, let $S$ be the subset of row indices we iterate on. Then, the next ARK iterate for $\b w$ is 
    \begin{align*}
    &\b w_{s_{i+1} + (j+1)}^{(i+1)} =  \mat{A}_S^\top  ( \mat{A}_S  \mat{A}_S^\top  + \lambda \mat{I})^{-1}\left ( \mat{A}_S \left[\b x_{s_{i+1} + j}^{(i+1)} - \b x_{s_{i+1}}\right] - \left[\b b_S - \mat{A}_S \b x_{s_{i+1}} \right] \right)
    \\ &\qquad =  \mat{A}_S^\top  ( \mat{A}_S  \mat{A}_S^\top  + \lambda \mat{I})^{-1}\left ( \mat{A}_S\b x_{s_{i+1} + j}^{(i+1)} - \b b_S \right) \\ 
    &\qquad =  \mat{A}_S^\top  ( \mat{A}_S  \mat{A}_S^\top  + \lambda \mat{I})^{-1}( \mat{A}_S \b x_{s_{i+1} + j} - \b b_S)  = \b w_{s_{i+1} + j}.
    \end{align*} 
\end{proof}

\subsection{Proof of \cref{lem:mu_nu}} \label{sec:mu_nu}

\begin{proof} [Proof of \cref{lem:mu_nu}]
First note that since the probability of choosing row $i$ is $\f {\norm{\b a_i}^2 + \lambda } {\norm{\mat A}_{\rm F}^2 + n\lambda}$, we get
$$ \overline {\b P}_{\lambda} = \m E[\b P_{\lambda, i}] = \sum_{i = 1}^n \f {\norm{\b a_i}^2 + \lambda } {\norm{\mat A}_{\rm F}^2 + n\lambda} \f {\b a_i \b a_{i}^\top }{\norm{\b a_{i}}^2 + \lambda} = \f {\mat{A}^\top  \mat A}{\norm{\mat A}_{\rm F}^2 + n \lambda}.$$
Observe
\begin{align*}
    \mu &= \lambda_{\rm min}^+(\overline{\b P}_\lambda)
     = 
    \lambda_{\rm min}^+\left (\f {\mat{A}^\top  \mat A}{\norm{\mat A}_{\rm F}^2 + n \lambda}\right) 
    = \f {\sigma_{\rm min}^2(\mat A)} { \norm{\mat A}_{\rm F}^2 + n \lambda}.
\end{align*}

For $\nu$, since we are considering the case where $\mat A$ is square and invertible, the pseudoinverses in the definition of $\nu$ are simply inverses. Let $\mat S = e_i^\top $ be the row vector so that $\mat S \mat A = \b a_i^\top$ is the $i$th row of $\mat A.$ Furthermore, let $f = \norm{\mat A}_{\rm F}^2 +n \lambda$, and let $f_i = \norm{\b a_i}^2 + \lambda$; thus, $f = \sum_{i = 1}^n f_i.$ Then,
\begin{align*}
    {\b P}_{\lambda, i} \overline{\b P}_\lambda^{\dagger}  \b P_{\lambda, i} &= 
    \f {\b a_i \b a_{i}^\top }{\norm{\b a_{i}}^2 + \lambda}
    \left( \f {\mat{A}^\top  \mat A}{\norm{\mat A}_{\rm F}^2 + n \lambda} \right )^{-1}
    \f {\b a_i \b a_{i}^\top }{\norm{\b a_{i}}^2 + \lambda} = \f {\b a_i \b a_{i}^\top }{f_i}
    \left( \f {\mat{A}^\top  \mat A}{f} \right )^{-1}
    \f {\b a_i \b a_{i}^\top }{f_i} \\ 
    &= \f {\mat A^\top  \mat S^\top  \mat S \mat A }{f_i}
    \left( \f {\mat{A}^\top  \mat A}{f} \right )^{-1}
    \f {\mat A^\top  \mat S^\top  \mat S \mat A }{f_i} = \f f {f_i^2} \mat A^\top  \mat S^\top  \mat S \mat A (\mat A^\top  \mat A)^{-1} \mat A^\top  \mat S^\top  \mat S \mat A \\ 
    &= \f f{ f_i^2} \mat A^\top  \mat S^\top  \mat S \mat S^\top  \mat S \mat A = \f f{f_i^2} \b a_i \b a_i^\top .
\end{align*}
Therefore,
\begin{align*}
    \nu &= \lambda_{\rm max}(\m E[( \overline{\b P}_\lambda^{\dagger/2} \b P_{\lambda, i} \overline {\b P}_\lambda^{\dagger/2})^2]) = \lambda_{\rm max}( \overline{\b P}_\lambda^{-1/2} \m E [  {\b P}_{\lambda, i} \overline{\b P}_\lambda^{-1}  \b P_{\lambda, i} ] \overline{\b P}_\lambda^{-1/2} ) \\ 
    &= \lambda_{\rm max} \left (\overline{\b P}_\lambda^{-1/2} \m E\left [ \f f {f_i^2} \b a_i \b a_i^\top  \right] \overline{\b P}_\lambda^{-1/2} \right) = \lambda_{\rm max} \left( \overline{\b P}_{\lambda}^{-1/2} \cdot  \sum_{i = 1}^n \f {f_i} f \f f {f_i^2} \b a_i \b a_i^\top  \cdot \overline{\b P}_\lambda^{-1/2} \right)  \\ 
    &= \lambda_{\rm max} \left( \overline{\b P}_{\lambda}^{-1/2} \cdot  \sum_{i = 1}^n \f {\b a_i \b a_i^\top } {f_i} \cdot \overline{\b P}_\lambda^{-1/2} \right). 
\end{align*}
Now, recall that when $\mat{X}$ and $\mat{Y}$ are positive semi-definite matrices, then $c_1 \mat{X} + c_2 \mat{Y} \preceq \mat{X} + \mat{Y}$ for any $0 \le c_1, c_2 \le 1.$ Also, if $\mat{Z}$ is a symmetric matrix, then $\mat{X} \preceq \mat{Y}$ implies $\mat{ZXZ} \preceq \mat{ZYZ}$. Combining these facts, we see that 
$$ \overline{\b P}_{\lambda}^{-1/2} \cdot  \sum_{i = 1}^n \f {\b a_i \b a_i^\top } {f_i} \cdot \overline{\b P}_\lambda^{-1/2} \preceq 
\overline{\b P}_{\lambda}^{-1/2} \cdot  \sum_{i = 1}^n \f {\b a_i \b a_i^\top } { \min_j f_j} \cdot \overline{\b P}_\lambda^{-1/2} = \f 1 {\min_j f_j} \overline{\b P}_{\lambda}^{-1/2} \cdot  \mat A^\top  \mat A \cdot \overline{\b P}_\lambda^{-1/2} .$$
So, 
\begin{align*}
    \nu  &= \lambda_{\rm max} \left( \overline{\b P}_{\lambda}^{-1/2} \cdot  \sum_{i = 1}^n \f {\b a_i \b a_i^\top } {f_i} \cdot \overline{\b P}_\lambda^{-1/2} \right) \le \lambda_{\rm max} \left(  \f 1 {\min_j f_j} \overline{\b P}_{\lambda}^{-1/2} \cdot \mat A^\top  \mat A \cdot \overline{\b P}_\lambda^{-1/2}  \right ) \\ 
    &= \f 1 {\min_j f_j} \lambda_{\rm max} \left( f \cdot (\mat A^\top  \mat A)^{-1/2} \cdot \mat A^\top  \mat A \cdot (\mat A^\top  \mat A)^{-1/2} \right) = \f f { \min_j f_j}.
\end{align*}
\end{proof}

\bibliographystyle{siamplain}
\bibliography{stability}

\end{document}